\documentclass[11pt]{amsart}
 
\usepackage{amsmath, amsthm, amssymb, amsfonts}

\usepackage{array}

\usepackage{soul}

\usepackage{graphicx,color}
\usepackage{tikz}
\usetikzlibrary{matrix,arrows,decorations.pathmorphing,quotes}
\usetikzlibrary{fit}
\tikzset{%
  highlight/.style={rectangle,rounded corners,fill=red!15,draw,
    fill opacity=0.5,thick,inner sep=0pt}
}
\newcommand{\tikzmark}[2]{\tikz[overlay,remember picture,
  baseline=(#1.base)] \node (#1) {#2};}
\newcommand{\Highlight}[1][submatrix]{%
    \tikz[overlay,remember picture]{
    \node[highlight,fit=(left.north west) (right.south east)] (#1) {};}
}

\newcommand{\bE}{\mathbf{E}}
\newcommand{\bI}{\mathbf{I}}

\newcommand{\bc}{\mathbf{c}}

\newcommand{\uc}{\underline{c}}
\newcommand{\ue}{\underline{e}}
\newcommand{\um}{\underline{m}}
\newcommand{\un}{\underline{n}}
\newcommand{\ut}{\underline{t}}
\newcommand{\umu}{\underline{\mu}}
\newcommand{\unu}{\underline{\nu}}

\newcommand{\NN}{\mathbb{N}}
\newcommand{\ZZ}{\mathbb{Z}}
\newcommand{\RR}{\mathbb{R}}
\newcommand{\CC}{\mathbb{C}}
\newcommand{\PP}{\mathbb{P}}

\newcommand{\TT}{\mathbb{T}}

\newcommand{\cA}{\mathcal{A}}
\newcommand{\cC}{\mathcal{C}}
\newcommand{\cO}{\mathcal{O}}
\newcommand{\cE}{\mathcal{E}}

\newcommand{\cK}{\mathcal{K}}

\newcommand{\cH}{\mathcal{H}}
\newcommand{\cT}{\mathcal{T}}

\renewcommand{\l}{\ell}

\DeclareMathOperator{\Cl}{Cl}
\DeclareMathOperator{\Div}{div}
\DeclareMathOperator{\coker}{coker}
\DeclareMathOperator{\codim}{codim}

\DeclareMathOperator{\cone}{cone}
\DeclareMathOperator{\Hom}{Hom}

\DeclareMathOperator{\im}{Im}

\DeclareMathOperator{\Spec}{Spec}
\DeclareMathOperator{\rk}{rk}

\DeclareMathOperator{\reg}{reg}
\DeclareMathOperator{\ri}{r.i.}

\DeclareMathOperator{\supp}{supp}

\theoremstyle{definition}
\newtheorem{definition}{Definition}[section]
\newtheorem{example}[definition]{Example}
\newtheorem{remark}[definition]{Remark}
\newtheorem{notation}[definition]{Notation}

\theoremstyle{plain}
\newtheorem{proposition}[definition]{Proposition}
\newtheorem{corollary}[definition]{Corollary}
\newtheorem{theorem}[definition]{Theorem}
\newtheorem{lemma}[definition]{Lemma}

\newcommand{\red}[1]{{\color{red} \sf $\clubsuit$  [#1]}}

\makeatletter
\@namedef{subjclassname@2020}{\textup{2020} Mathematics Subject Classification}
\makeatother

\begin{document}

\title{Multigraded Castelnuovo-Mumford regularity via Klyachko filtrations}

\author[Rosa M. Mir\'o-Roig]{Rosa M. Mir\'o-Roig}
\address{Department de matem\`{a}tiques i Inform\`{a}tica, Universitat de Barcelona, Gran Via de les Corts Catalanes 585, 08007 Barcelona,
Spain}
\email{miro@ub.edu}

\author[Marti Salat-Molt\'o]{Mart\'i Salat-Molt\'o}
\address{Department de matem\`{a}tiques i Inform\`{a}tica, Universitat de Barcelona, Gran Via de les Corts Catalanes 585, 08007 Barcelona,
Spain}
\email{marti.salat@ub.edu}

\begin{abstract}
In this paper, we consider $\ZZ^{r}-$graded modules on the $\Cl(X)$ $-$graded Cox ring $\CC[x_{1},\dotsc,x_{r}]$ of a smooth complete toric variety $X$. Using the theory of Klyachko filtrations in the reflexive case, we construct a collection of lattice polytopes codifying the multigraded Hilbert function of the module. We apply this approach to reflexive $\ZZ^{s+r+2}$-graded modules over non-standard bigraded polynomial rings $\CC[x_{0},\dotsc,$ $x_{s},y_{0},\dotsc,y_{r}]$. In this case, we give sharp bounds for the multigraded regularity index of their multigraded Hilbert function, and a method to compute their Hilbert polynomial.
\end{abstract}


\thanks{Acknowledgements: The first author was partially supported by MTM2016-78623-P. The second author is partially supported by  MDM-2014-0445-18-2.}

\maketitle

\tableofcontents

\markboth{}{}


\section{Introduction}

Let $G$ be a finitely generated abelian group and $\alpha_{1},\dotsc,\alpha_{r}\in G$. The polynomial ring $R=\CC[x_1,\dotsc,x_{r}]$ is naturally endowed with a $G-$grading setting $\deg(x_{i})=\alpha_{i}$ for $1\leq i\leq r$, $\deg(a)=0$ for any $a\in k$ and extending algebraically. When $G=\ZZ$ and $\alpha_{1}=\dotsb=\alpha_{r}=1$, we recover the standard $\ZZ-$graded polynomial ring. In the past few decades, there has been an increasing interest on general $G-$graded rings and their homogeneous modules (see for instance \cite{Ara-Cro-Neg, Big, Bot-Cha, Cro, Hoa-Hyry, Hoa-Tru, Mac-Smi1, Mac-Smi2, Tru-Ver}). Many of them focus on either of the following two points of view: (1) the Cox ring of a toric variety $X$, which is a $\Cl(X)-$graded polynomial ring $R$ where $\Cl(X)$ denotes the class group of $X$; and (2) the $\ZZ^{r}-$graded (or {\em multigraded}) polynomial ring, which is the finest way to grade $R$. There are multiple connections between both standpoints, for instance a multigraded module is automatically homogeneous with respect to any $G-$grading of $R$.

In this paper, we focus on $\Cl(X)-$graded modules which are also $\ZZ^{r}$--homogeneous. The advantage of this approach is twofold. First, it allows us to study $\ZZ^{r}-$graded modules using combinatorial tools arising from toric geometry. On the other way round, the whole class of $\Cl(X)-$graded modules is in general too wide, even in the rank $1$ case. So very often further conditions must be imposed to obtain information on invariants like the Hilbert function, the regularity or the minimal free resolution. In this note, we provide sharp bounds for the multigraded index regularity and a method to compute explicitly the multigraded Hilbert function and Hilbert polynomial for $\Cl(X)-$graded reflexive modules of any rank $\l\geq 1$ which are $\ZZ^{r}-$homogeneous. 


Let $X$ be an $n-$dimensional toric variety with fan $\Sigma$, and we set $r=|\Sigma(1)|$ the number of rays of $\Sigma$. The Cox ring of $X$ is a polynomial ring $R=\CC[x_{1},\dotsc,x_{r}]$ graded by the class group $\Cl(X)$ of $X$ and it allows to establish a close bond between $G-$graded commutative algebra and toric geometry. For instance, any $\Cl(X)-$graded module $E$ corresponds to a quasi-coherent sheaf $\widetilde{E}$ on $X$ related to $E$ by the following exact sequence:
\begin{equation}
0\rightarrow H^{0}_{B}(E)\rightarrow E\rightarrow H^{0}_{\ast}(X,\widetilde{E})\rightarrow H^{1}_{B}(E)\rightarrow 0,
\end{equation}
where $H^{0}_{\ast}(X,\widetilde{E}):=\bigoplus_{\alpha\in\Cl(X)}H^{0}(X,\widetilde{E}(\alpha))=:\Gamma E$ is the {\em $B-$saturation} of $E$ and $B\subset R$ the so-called {\em irrelevant ideal}. In particular, we can compute the multigraded Hilbert function $h_{\Gamma E}$ of $\Gamma E$, and the multigraded Hilbert polynomial of $E$, $P_{E}$, using the sheaf cohomology of $\widetilde{E}$: for any $\alpha\in\Cl(X)$, $h_{\Gamma E}(\alpha)=H^{0}(X,\widetilde{E})$, and $P_{E}(\alpha)=\chi(\widetilde{E}(\alpha))=\sum_{i=0}^{n}(-1)^{i}H^{i}(X,\widetilde{E}(\alpha))$.
Later on \cite{Batyrev-Cox}, it was proved that when $E$ is also $\ZZ^{r}-$homogeneous, the sheaf $\widetilde{E}$ is {\em equivariant}. Equivariant sheaves were previously studied in \cite{Kly89} and \cite{Kly91}. In these papers, Klyachko introduced a classification of torsion-free equivariant sheaves in terms of {\em multifiltrations} of vector spaces parametrized by the cones in $\Sigma$. In the past few years, the strategy of Klyachko has been used to study equivariant quasi-coherent sheaves from diverse geometrical standpoints (see for instance \cite{Her-Mus-Pay, Pay, Perling}). However, it has not been applied in the context of $G-$graded commutative algebra. From this perspective, we tackle two main problems in the context of $G-$graded commutative algebra: to bound the multigraded Castelnuovo-Mumford regularity and the study of the multigraded Hilbert function. Introduced in \cite{Mac-Smi1}, Multigraded Castelnuovo-Mumford regularity of $\Cl(X)-$graded $R-$modules has received a lot of attention in the last decades, particularly regarding to $\Cl(X)-$graded ideals and their associated coordinate rings. See for instance \cite{Bot-Cha, Cha-Nem, Mac-Smi2, Sah-Sop}.

In \cite{Perling}, Perling formalized and generalized Klyachko's classification introducing the theory of $\Sigma-$families to describe more general quasi-coherent equivariant sheaves. We observe that this construction induces a decomposition of $\Gamma E$, which is particularly well behaved when $E$ is reflexive. By means of this decomposition, we define a collection of lattice polytopes parametrized by $\Cl(X)$ which codifies the multigraded Hilbert function $h_{\Gamma E}$ of $\Gamma E$. Asymptotically, this collection of polytopes
is used to compute  the Hilbert polynomial $P_{E}$ of $E$. We call $\ri(E)=\{\alpha\in \Cl(X)\;|\;h_{E}(\alpha)=P_{E}(\alpha)\}$ the multigraded regularity index of the Hilbert function of $E$. The multigraded regularity index $\ri(\Gamma E)$ is related to the multigraded Castelnuovo--Mumford regularity. In Section \ref{Section:Applications}, we use this methods to find sharp bounds for the multigraded regularity index.

\vspace{2mm}
Next we explain how this paper is organized. Section \ref{Section:Preliminaries} contains all the preliminary results we use in this work. In subsection \ref{Section:Preliminaries graded modules},  we collect the basic definitions and results of $G-$graded rings and modules, and toric varieties. In Subsection \ref{Section:Preliminaries Klyachko filtrations for modules}, we summarize Klyachko's classification following \cite{Perling}, making particular emphasis on the classification of equivariant reflexive sheaves. 

The main body of the article is gathered in the remaining two sections. In Section \ref{Section:Multigraded regularity}, we use the theory of $\Sigma-$families to introduce the collection of polytopes $\{\Omega_{\um}(\ut)\}_{\ut\in\Cl(X)}$. The main result of this section is Proposition \ref{Proposition:Hilbert regularity}, showing that the lattice points of $\{\Omega_{\um}(\ut)\}_{\ut\in\Cl(X)}$ codify the multigraded Hilbert function of $\Gamma E$, as well as bound its multigraded support. In Section \ref{Section:Applications}, we apply this construction to multigraded reflexive modules over the Cox ring of a toric variety with $1-$splitting fan (see Definition \ref{Definition:Splitting}). Namely, let $0\leq a_{1}\leq\dotsb\leq a_{r}$ be integers and $R=\CC[x_{0},\dotsc,x_{s}, y_{0},\dotsc,y_{r}]$ with $\deg(x_{i})=(1,0)$ for $0\leq i\leq s$, $\deg(y_{0})=(0,1)$ and $\deg(y_{j})=(-a_{j},1)$ for $1\leq j\leq r$. We provide upper and lower bounds for the support of a rank $\l\geq1$ multigraded reflexive $R-$module $E$. In Theorem \ref{Theorem:Bound hilbert polynomial}, we apply the previously described tools to give sharp bounds of $\ri(\Gamma E)$. The proof of this result is constructive, thus describing an explicit way to compute $P_{E}$, the Hilbert polynomial of $E$. We end the paper with an example which shows that the bounds given in Theorem \ref{Theorem:Bound hilbert polynomial} are sharp.

\section{Preliminaries}\label{Section:Preliminaries}
\subsection{Modules graded by abelian groups}
\label{Section:Preliminaries graded modules}

Let us fix a ring $R$ and an abelian group $G$. A $R-$algebra $A$ is {\em $G-$graded} if it can be decomposed as a direct sum of $R-$submodules $A=\bigoplus_{g\in G}A_{g}$ such that $A_{g}\cdot A_{h}\subset A_{g+h}$ for any $g,h\in G$. A morphism of $G-$graded $R-$algebras $\psi:A\rightarrow B$ is $G-$graded if $\psi(A_{g}) \subset B_{g}$ for any $g\in G$. Similarly, an $A-$module $E$ is $G-$graded if it decomposes as a direct sum of submodules $E=\bigoplus_{g\in G} E_{g}$ such that $A_{g}\cdot E_{h}\subset E_{g+h}$.
Given $\pi:G\rightarrow H$ a surjective morphism of abelian groups, any $G-$graded $R-$algebra $A$ is also endowed with an $H-$grading:
\[
A=\bigoplus_{h\in H}A_{h},\quad\text{with}\quad A_{h}:=\bigoplus_{\substack{g\in G\\ \pi(g)=h}}A_{g}\quad\text{for any}\quad h\in H.
\]
Moreover, we have the following lemma:

\begin{lemma}\label{Lemma:Grading and exact sequence}
Let $0\rightarrow K \xrightarrow{\phi} G\xrightarrow{\pi} H\rightarrow 0$ be an exact sequence of abelian groups, and $A$ a $G-$graded $R-$algebra. Then, for any $h\in H$, $A_{h}$ is a $K-$graded $R-$module.
\end{lemma}
\begin{proof}
We fix $h\in H$ and let $x, y\in G$ be such that $\pi(x)=\pi(y)=h$. By exactness $x-y\in \im(\phi)$, and $\pi^{-1}(h)=\{x + \phi(k)\;\mid\;k\in K\}=\{y + \phi(k)\;\mid\;k\in K\}$. On the other hand, we consider the $R-$module $A_{h}$ as before and we have
\[
A_{h}=\bigoplus_{g\in \pi^{-1}(h)}A_{g}=\bigoplus_{k\in K}A_{x+\phi(k)}.
\]
Notice that this decomposition does not depend (up to permutation) on the chosen preimage $x\in G$, and it structures $A_{h}$ as a $K-$graded $R-$module.
\end{proof}

\begin{example}\label{Example:Graded polynomial rings}
\begin{itemize}
\item[(i)] The polynomial ring $A=R[x_{1},\dotsc,x_{n}]$ is $\ZZ-$graded as well as it is $\ZZ^{n}-$graded.
\item[(ii)] Let us consider $G=\ZZ/2$, and set $A_{\overline{0}}:=R\langle x_{1}^{\alpha_{1}}\dotsb x_{n}^{\alpha_{n}}\;\mid\;\alpha_{1}+\dotsb+\alpha_{n}\equiv 0 \mod 2\rangle$ and $A_{\overline{1}}:=R\langle x_{1}^{\alpha_{1}}\dotsb x_{n}^{\alpha_{n}}\;\mid\;\alpha_{1}+\dotsb+\alpha_{n}\equiv 1 \mod 2\rangle$. Then, $A=A_{\overline{0}}\oplus A_{\overline{1}}$ is a $\ZZ/2-$graded $R-$algebra, such that $\deg(x_{i})=\overline{1}$.
\item[(iii)] More in general, for any abelian group $G$ and $w_{1},\dotsc, w_{n}\in G$, we set
$A_{g}:=R\langle x_{1}^{\alpha_{1}}\dotsb x_{n}^{\alpha_{n}}\;\mid\;
\alpha_{1} w_{1}+\dotsb+\alpha_{n} w_{n} = g\rangle.$
This structures $A$ as a $G-$graded $R-$algebra. Notice that this $G-$grading is determined by setting $\deg(x_{i})=w_{i}$. In particular, for any element $g\in G$, we have the following decomposition
\[
A_{g}
\hspace{3mm}
=
\hspace{-7mm}
\bigoplus_{
\substack{
	(\alpha_{1},\dotsc,\alpha_{n})\in\ZZ^{n}\\
	\alpha_{1} w_{1}+\dotsb+\alpha_{n} w_{n} =g
	}
}\hspace{-8mm}
A_{(\alpha_{1},\dotsc,\alpha_{n})}.
\]
This shows that the $\ZZ^{n}-$graduation is the finest grading for the polynomial algebra $R[x_{1},\dotsc,x_{n}]$.
\end{itemize}
\end{example}

We say that a $R-$algebra $A$ (respectively an $A-$module) is $(G,H)-$graded if it is both a $G-$graded and $H-$graded. Finally, the following lemma shows that the localization by homogeneous elements respects the graduation
\begin{lemma}\label{Lemma:localization}
Let $A$ be a $G-$graded $R-$algebra, and $f\in A$ a homogeneous element of degree $d$. Then, the localized $R-$algebra $A_{f}$ at $f$ is $G-$graded.
\end{lemma}
\begin{proof}
For any $g\in G$, we define the $R-$submodule $(A_{f})_{g}:=\{\frac{a}{f^{n}}\;\mid\;\deg(a)-nd=g\}\subset A_{f}$. It is straightforward that $A_{f}=\bigoplus_{g\in G} (A_{f})_{g}$ and it structures $A_{f}$ as a $G-$graded $R-$algebra.
\end{proof}

One important family of graded algebras by abelian groups come from the theory of toric varieties. Many results stated in the sequel hold more generally for simplicial toric varieties. However, for our purposes we focus on the smooth case. For a general reference of toric varieties we refer to \cite{CLS}. Let $X$ be an $n-$dimensional smooth complete toric variety with torus $\TT_{N}\cong(\CC^{\ast})^{n}$, associated to a fan $\Sigma\subset N\otimes\RR\cong\RR^{n}$, where $N\cong\ZZ^{n}$ is a lattice. We denote by $\Sigma(k)$ the set of $k-$dimensional cones. We call {\em rays} the cones $\rho\in\Sigma(1)$ and we set $n(\rho)\in N$ the first lattice point of $\rho\cap N$. We denote by $M=\Hom(N,\ZZ)\cong\ZZ^{n}$ its {\em character lattice} and for $m\in M$, we set $\chi^{m}:\TT_{N}\rightarrow\CC^{\ast}$ the corresponding rational function. For any cone $\sigma\in\Sigma$, let $\sigma^{\vee}$ be its {\em dual cone}, let $S_{\sigma}:=\sigma^{\vee}\cap M$ be the associated semigroup of characters, $\CC[S_{\sigma}]$ the corresponding $\CC-$algebra, and $U_{\sigma}=\Spec(\CC[S_{\sigma}])\subset X$. The affine subvarieties $U_{\sigma}$ are related as follows: if $\sigma\in \Sigma$ and $\tau\prec\sigma$ is a face, then there is a character $m\in M$ such that $S_{\tau}=S_{\sigma}+\ZZ\langle m\rangle$. As a consequence, there is an inclusion $U_{\tau}\hookrightarrow U_{\sigma}$ given by the natural morphism of $\CC-$algebras $\CC[S_{\sigma}]\hookrightarrow \CC[S_{\sigma}]_{\chi^{m}}=\CC[S_{\tau}]$. On the other hand, there is a bijection between $\TT_{N}-$orbits $O(\sigma)$ of $X$ and cones $\sigma\in\Sigma$ such that $\dim(O(\sigma))=\codim(\sigma)$, and we set $V(\sigma)=\overline{O(\sigma)}$. For the particular case of a ray $\rho\in\Sigma(1)$, $V(\rho)$ is an invariant Weil divisor, which we denote by $D_{\rho}$. The class group $\Cl(X)$ of $X$ is determined by the invariant Weil divisors.
\begin{proposition}\label{Proposition:Class group exact sequence}
We have the exact sequence
\begin{equation}\label{Eq:Class group exact sequence}
0\rightarrow M\xrightarrow{\phi}\bigoplus_{\rho\in\Sigma(1)}\ZZ D_{\rho}\xrightarrow{\pi} \Cl(X)\rightarrow 0,
\end{equation}
where $\phi(m)=\Div(\chi^{m})=\sum_{\rho\in\Sigma(1)}\langle m, n(\rho)\rangle D_{\rho}$, for any character $m\in M$; and for any invariant Weil divisor $D$, $\pi(D)=[D]$ is its class in $\Cl(X)$. In particular, $\Cl(X)$ is a finitely generated abelian group, and we set $c:=\rk\Cl(X)$.
\end{proposition}
\begin{proof}
See, for instance, \cite[Theorem 4.1.3]{CLS}.
\end{proof}
 
Let $R=\CC[x_{\rho}\;|\;\rho\in\Sigma(1)]$ be a polynomial ring in $|\Sigma(1)|$ variables. The {\em Cox ring} of $X$ is the $\CC-$algebra $R$ endowed with a grading given by the class group $\Cl(X)$ of $X$. More precisely,  we set $\deg(x_{\rho}):=[D_{\rho}]\in\Cl(X)$, for each ray $\rho\in\Sigma(1)$ and extend the grading as in Example \ref{Example:Graded polynomial rings} (iii). Abusing notation we write $R=\CC[x_{1},\dotsc,x_{r}]$ whenever $\Sigma(1)=\{\rho_{1},\dotsc,\rho_{r}\}$ is the (ordered) set of rays of $\Sigma$.  On the other hand, for any cone $\sigma\in\Sigma$, set 
\[
x^{\hat{\sigma}}:=\prod_{\rho_{i}\in \Sigma(1)\setminus \sigma(1)}x_{i},\quad\text{and}\quad B:=\langle x^{\hat{\sigma}}\;\mid\;\sigma\in\Sigma\rangle.
\]
$B$ is called the {\em irrelevant ideal} and it holds that $B=\langle x^{\hat{\sigma}}\;\mid\;\sigma\in\Sigma_{\max}\rangle$. 
\begin{example}\label{Example:Projective space}
$\PP^{n}$ is a toric variety of dimension $n$. Let $\{e_{1},\dotsc,e_{n}\}$ be a basis of $N=\ZZ^{n}$. The fan $\Sigma$ associated to $\PP^{n}$ has $n+1$ rays: $\rho_{0}=\cone(-e_{1}-\dotsb-e_{n})$ and $\rho_{i}=\cone(e_{i})$ for $1\leq i\leq n$; and $n+1$ maximal cones $\sigma_{0}:=\cone(e_{1},\dotsc,e_{n})$ and $\sigma(i):=\cone(\{e_{j}\;\mid\;j\neq i\}\cup\{-e_{1}-\dotsb-e_{n}\})$ for $1\leq i\leq n$. Its  associated Cox ring is $\CC[x_{0},\dotsc,x_{n}]$ with $\deg(x_{i})=1$ for $0\leq i\leq n$, and the irrelevant ideal is $B(\Sigma)=\langle x_{0},\dotsc,x_{n}\rangle$.
\end{example}

\begin{example}\label{Example:Hirzebruch} For $a\geq0$, the Hirzebruch surface $\cH_{a}=\PP(\cO_{\PP^{1}}\oplus\cO_{\PP^{1}}(a))$ is a toric surface.
Let $N=\ZZ^{2}$ be a lattice, denote $\{e,f\}$ its standard basis and set $u_{0}:=-e+af$, $u_{1}:=e$, $v_{0}:=-f$ and $v_{1}:=f$. The fan $\Sigma$ associated to $\cH_{a}$ has four rays $\rho_{0}=\cone(u_{0})$, $\rho_{1}=\cone(u_{1})$, $\eta_{0}=\cone(v_{0})$ and $\eta_{1}=\cone(v_{1})$; and four maximal cones cones $\sigma_{00}=\cone(u_{1},v_{1})$, $\sigma_{01}=\cone(u_{1},v_{0})$, $\sigma_{10}=\cone(u_{0},v_{1})$ and $\sigma_{11}=\cone(u_{0},v_{0})$. Its Cox ring is $\CC[x_{0},x_{1},y_{0},y_{1}]$ with $\deg(x_{0})=\deg(x_{1})=(1,0)$, $\deg(y_{0})=(0,1)$ and $\deg(y_{1})=(-a,1)$; and its irrelevant ideal is $B(\Sigma)=\langle x_{1}y_{1}, x_{1}y_{0}, x_{0}y_{1}, x_{0}y_{0} \rangle$.
\end{example}

For any cone $\sigma\in \Sigma$, the localization of $R$ at $x^{\hat{\sigma}}$ is by Lemma \ref{Lemma:localization} a $\Cl(X)-$graded algebra $R_{x^{\hat{\sigma}}}$. Furthermore, there is an isomorphism $\CC[S_{\sigma}]\cong (R_{x^{\hat{\sigma}}})_{0}$, sending $\chi^{m}$ to the monomial $x_{1}^{\langle m, \rho_{1}\rangle}\dotsb x_{r}^{\langle m, \rho_{r}\rangle}$ for any $m\in S_{\sigma}$ and extending algebraically. The following results summarize the relation between $\Cl(X)-$graded $R-$modules and quasi-coherent sheaves on $X$.

\begin{proposition}\label{Proposition:line bundles}
 Fix $\alpha \in \Cl(X)$. Then, there is a natural isomorphism $R_{\alpha} \cong \Gamma(X,\mathcal{O}_{X}(D))$ for any Weil divisor $D = \sum_{\rho} a_{\rho}D_{\rho}$ such that $\alpha = [D]$
\end{proposition}
\begin{proof}
See for instance \cite[Proposition 5.3.7]{CLS}.
\end{proof}

\begin{proposition} \label{Proposition:quasi-coherent sheaves}
\begin{itemize}
\item[(i)] If $E$ is a $\Cl(X)-$graded $R$-module, there is a quasi-coherent sheaf $\widetilde{E}$ on $X$ such that $\Gamma(U_{\sigma},\widetilde{E}) = (E_{x^{\hat{\sigma}}})_{0}$, for any $\sigma \in \Sigma$.
\item[(ii)] Conversely, if $\cE$ is a quasi-coherent sheaf on $X$, there is a $\Cl(X)-$gra\-ded $R-$module such that $\widetilde{E}=\cE$. In particular, $\widetilde{E}$ is coherent if and only if $E$ is finitely generated. 
\item[(iii)] $\widetilde{E}=0$ if and only if $B^{l}E=0$ for all $l\gg 0$.

\item[(iv)] There is an exact sequence of $\Cl(X)-$graded modules
\[
0\rightarrow H^{0}_{B}(E)\rightarrow E \rightarrow H^{0}_{\ast}(X,\widetilde{E})\rightarrow H^{1}_{B}(E)\rightarrow0
\]
\item[(v)] $H^{i+1}_{B}(E)\cong H^{i}_{\ast}(X,\widetilde{E}(\alpha))=\bigoplus_{\alpha\in\Cl(X)}H^{i}(X,\widetilde{E}(\alpha))$.
\end{itemize}
\end{proposition}
\begin{proof}
See for instance \cite[Proposition 5.3.3, Proposition 5.3.6 and Proposition 5.3.10]{CLS} for (i)-(iii); and \cite[Proposition 2.3]{Eis-Mus-Sti} for (iv)-(v).
\end{proof}

The module $\Gamma E = H^{0}_{\ast}(X,\widetilde{E})$ is called the {\em $B-$saturation} of $E$. We say that $E$ is {\em $B-$saturated} if $E\cong \Gamma E$, or equivalently if $H^{0}_{B}(E)=H^{1}_{B}(E)=0$. If $H^{0}_{B}(E)=(0:_{E}B^{\infty})=\bigcup_{i\geq0}(0:_{E}B^{i})=\bigcup_{i\geq0}\{f\in E\mid f B^{i}=0\}=0$, we say that $E$ is $B-$torsion free.


\subsection{Multigraded modules over Cox rings and Klyachko filtrations}
\label{Section:Preliminaries Klyachko filtrations for modules}

Hereafter, we fix a smooth complete toric variety $X$. Let $R=\CC[x_{1},\dotsc,x_{r}]$ be its associated $\Cl(X)-$graded Cox ring, where $r=|\Sigma(1)|$, and $B$ the corresponding irrelevant ideal. In this subsection we focus on $(\Cl(X), \ZZ^{r})-$graded $R-$modules and their associated quasi-coherent sheaves which are equivariant. 
We start recalling the notion of $\Sigma-$family used to describe equivariant sheaves and later we pay attention to {\em reflexive} sheaves, which will play a central role in the forthcoming sections.
\begin{definition}
For any $t\in\TT_{N}$, let $\mu_{t}:X\rightarrow X$ be the morphism given by the action of $\TT_{N}$ on $X$. A quasi-coherent sheaf $\cE$ on $X$ is {\em equivariant} if there is a family of isomorphisms $\{\phi_{t}:\mu_{t}^{\ast}\cE\cong\cE\}_{t\in\TT_{N}}$ such that $\phi_{t_{1}\cdot t_{2}}=\phi_{t_{2}}\circ\mu_{t_{2}}^{\ast}\phi_{t_{1}}$ for any $t_{1},t_{2}\in\TT_{N}$.
\end{definition}
In \cite{Batyrev-Cox}, Batyrev and Cox proved the following result:
\begin{proposition}\label{Proposition:Multigraded-Equivariant}
Let $E$ be a $\Cl(X)-$homogeneous $R-$module. The quasi-coherent sheaf $\widetilde{E}$ is equivariant if and only if $E$ is also $\ZZ^{r}-$graded.
\end{proposition}
\begin{proof}
See for instance \cite[Proposition 4.17]{Batyrev-Cox}.
\end{proof}
In \cite{Kly89} and \cite{Kly91}, Klyachko observed that an equivariant torsion-free sheaf can be decomposed into a family of filtered vector spaces. This feature was further developed by Perling in \cite{Perling}. For seek of completeness, we recall the construction as presented in \cite{Perling}.
If $E$ is a $(\Cl(X), \ZZ^{r})-$graded $R-$module, by Lemma \ref{Lemma:Grading and exact sequence} and using the exact sequence (\ref{Eq:Class group exact sequence}), each homogeneous component of degree $\alpha\in\Cl(X)$ is $M-$graded. More precisely,
\[
E_{\alpha}=\bigoplus_{m\in M}E_{z+\phi(m)},\quad\text{for any}\quad z\in\pi^{-1}(\alpha).
\]
Notice that any monomial in $R$ is $(\Cl(X),\ZZ^{r})-$homogeneous. Thus, for any $\sigma\in \Sigma$ and by Lemma \ref{Lemma:localization}, we have an analogous decomposition for the localized $\CC[S_{\sigma}]-$module $E^{\sigma}:=(E_{x^{\hat{\sigma}}})_{0}$.
\begin{equation}\label{Eq:Sigma family decomposition}
E^{\sigma}=\bigoplus_{m\in M}(E_{x^{\hat{\sigma}}})_{\phi(m)}=:\bigoplus_{m\in M}E^{\sigma}_{m}.
\end{equation}
Setting $\cE:=\widetilde{E}$ we have $E^{\sigma}\cong \Gamma(U_{\sigma},\cE)$. Recall that $S_{\sigma}$ induces a preorder on $M$: for any $m,m'\in M$ we say that $m\leq_{\sigma} m'$ if and only if $m'-m\in S_{\sigma}$, or equivalently if $\langle m' - m, u\rangle \geq0$ for all $u\in\sigma$.
For any two characters $m\leq_{\sigma}m'$ we define $\chi^{\sigma}_{m,m'}:E^{\sigma}_{m}\rightarrow E^{\sigma}_{m'}$ to be the multiplication by $\chi^{m'-m}\in \CC[S_{\sigma}]$. We have $\chi^{\sigma}_{m,m}=1$, and $\chi^{\sigma}_{m,m''}=\chi^{\sigma}_{m',m''}\circ\chi^{\sigma}_{m,m'}$ for any $m\leq_{\sigma}m'\leq_{\sigma}m''$. In particular 
if $m\leq_{\sigma} m'$ and $m'\leq_{\sigma}m$, $\chi^{\sigma}_{m,m'}$ is an isomorphism.
We call $\hat{E}^{\sigma}:=\{E^{\sigma}_{m},\chi_{m,m'}^{\sigma}\}$ a {\em $\sigma-$family} (see \cite[Definition 4.2]{Perling}).

Now, let $\tau\prec\sigma$ be two cones in $\Sigma$. There is $m\in M$ such that localizing at $\chi^{m}$, we have the isomorphisms $\CC[S_{\tau}]\cong\CC[S_{\sigma}]_{\chi^{m}}$ and $E^{\tau}\cong E^{\sigma}_{\chi^{m}}$. Thus, we have a morphism $i^{\sigma \tau}:\CC[S_{\sigma}]\rightarrow\CC[S_{\tau}]$ inducing a morphism $i^{\sigma \tau}_{m'}:E^{\sigma}_{m'}\rightarrow E^{\tau}_{m'}$ for any character $m'\in M$. We call $\{\hat{E}^{\sigma}\}_{\sigma\in \Sigma}$ a {\em $\Sigma-$family} (see \cite[Definition 4.8]{Perling}). Moreover, the class of all $\Sigma-$families form a category and we have the following result.
\begin{proposition}\label{Proposition:Delta family and quasi-coherent sheaves}
Let $\Sigma$ be a fan. The category of $\Sigma-$families is equivalent to the category of equivariant quasi-coherent sheaves over $X$.
\end{proposition}
\begin{proof}
See \cite[Theorem 4.9]{Perling}.
\end{proof}

As a corollary, the $\Sigma-$family characterizes the saturation $\Gamma E$ of any $(\Cl(X),\ZZ^{r})-$graded module $E$. When the module is torsion-free, then $E^{\sigma}$ is a torsion-free $\CC[S_{\sigma}]-$module. In that case, every vector space in the $\Sigma-$family can be viewed as a subspace of $\CC^{\l}$ where $\l=\rk E$. More precisely we have the following result.
\begin{proposition}\label{Proposition:Torsionfrees}
Let $E$ be a torsion-free $(\Cl(X),\ZZ^{r})-$graded module of rank $\l$, and let $\{\hat{E}^{\sigma}\}$ be its associated $\Sigma-$family. The following holds:
\begin{itemize}
\item[(i)] For any $m'\leq_{\sigma} m$, the linear map $\chi_{m',m}^{\sigma}:E^{\sigma}_{m'}\rightarrow E^{\sigma}_{m}$ is injective.
\item[(ii)] For any $m\in M$, and any cones $\tau\prec\sigma$ in $\Sigma$, the linear map $i^{\sigma\tau}_{m}:E^{\sigma}_{m}\rightarrow E^{\tau}_{m}$ is injective.
\item[(iii)] There is a vector space $\bE\cong\CC^{\l}$ such that $E^{\{0\}}_{m}\cong \bE$ for any $m\in M$. 
\end{itemize}
Moreover, we have the following commutative diagram:
\vspace{-1mm}
\begin{center}
\begin{tikzpicture}
\matrix (M) [matrix of nodes, row sep=0.2cm, column sep=1.8cm, align=center,text width=0.8cm, text height=0.3cm, anchor=center]{
              & $E^{\{0\}}_{m'}$ & $E^{\sigma}_{m'}$ \\
$\bE$ &                &                 \\
              & $E^{\{0\}}_{m}$  & $E^{\sigma}_{m}.$  \\
};
\draw[->] (M-1-3) -- (M-3-3) node[midway,right]{\scriptsize{$\chi^{\sigma}_{m',m}$}};
\draw[->] (M-1-2) -- (M-3-2) 
	node[midway,right]{\scriptsize{$\chi^{\{0\}}_{m',m}$}} 
	node[midway,left]{\large{$\cong\quad\;$}};
\draw[->] (M-1-3) -- (M-1-2);
\draw[->] (M-3-3) -- (M-3-2);
\draw[->] (M-3-3) -- (M-3-2);
\draw[->] (M-3-3) -- (M-3-2);
\draw[->] (M-1-2) -- (M-2-1)
	node[midway, above]{\scriptsize{$\varphi_{m'}$}};
\draw[->] (M-3-2) -- (M-2-1)
	node[midway, below]{\scriptsize{$\varphi_{m}$}};
\end{tikzpicture}
\end{center}
\end{proposition}
\begin{proof}
See \cite[Section 4.4]{Perling}.
\end{proof}

Notice that any monomial ideal $I\hspace{-1mm}=\hspace{-1mm}\langle m_{1},\dotsc,m_{s}\rangle\hspace{-1mm}\subset\hspace{-1mm} R$ is a $(\Cl(X),\ZZ^{r})-$gra\-ded torsion-free $R-$module. Next example illustrates Proposition \ref{Proposition:Torsionfrees}.

\begin{example}
Let $R=\CC[x_{0},x_{1},x_2]$ be the Cox ring of $\PP^{2}$ with fan $\Sigma$ as in Example \ref{Example:Projective space}, and let $I=(x_{2}^2,x_{0}x_{2},x_{0}x_{1})$ be a monomial ideal. Let us compute the $\Sigma-$family associated to $I$. We present $I$ as follows:
\begin{equation}\label{Eq:Example monomial ideal}
R(0,0,-2)\oplus R(-1,0,-1)\oplus R(-1,-1,0)\xrightarrow{(x_{2}^2\;x_{0}x_{2}\;x_{0}x_{1})}I\rightarrow0.
\end{equation}
Next, we localize at $x^{\widehat{\{0\}}}=x_{0}x_{1}x_{2}$ and we set $R^{\{0\}}:=R_{x^{\widehat{\{0\}}}}$ the localized ring. For any multidegree $(\alpha_{0},\alpha_{1},\alpha_{2})\in\ZZ^{3}$, $R^{\{0\}}_{(\alpha_{0},\alpha_{1},\alpha_{2})}=\CC\langle x_{0}^{\alpha_{0}}x_{1}^{\alpha_{1}}x_{2}^{\alpha_{2}}\rangle$, the vector space spanned by the monomial $x_{0}^{\alpha_{0}}x_{1}^{\alpha_{1}}x_{2}^{\alpha_{2}}$. On the other hand, any character $m=(d_{1},d_{2})$ is embedded as $m=(-d_{1}-d_{2},d_{1},d_{2})$ in $\ZZ^{3}$ via the exact sequence (\ref{Eq:Class group exact sequence}). To compute $I^{\{0\}}_{m}$ we take the degree $m$ component of (\ref{Eq:Example monomial ideal}). This yields the following exact sequence of vector spaces
\[
R^{\{0\}}(0,0,-2)_{m}\oplus R^{\{0\}}\!(-1,0,-1)_{m}\oplus R^{\{0\}}\!(-1,-1,0)_{m}\xrightarrow{(x_{2}^2\;x_{0}x_{2}\;x_{0}x_{1})}I^{\{0\}}_{m}\rightarrow0.
\]
Thus, $I^{\{0\}}_{m}=\CC\langle x_{0}^{-d_{1}-d_{2}}x_{1}^{d_{1}}x_{2}^{d_{2}}\rangle$ and there are isomorphisms $\phi^{\{0\}}_{m}:I^{\{0\}}_{m}\rightarrow \CC\langle u_{1}\rangle=:\bI$. 
Let us now fix the ray $\rho_{0}\in\Sigma(1)$ and compute $I^{\rho_{0}}_{m}$ for any character $m=(d_{1},d_{2})\in\ZZ^{2}$. As before, we set $R^{\rho_{0}}:=R_{x^{\widehat{\rho_{0}}}}$ the localization at $x^{\widehat{\rho_{0}}}=x_{1}x_{2}$. Now, for any multidegree $(\alpha_{0},\alpha_{1},\alpha_{2})\in\ZZ^{3}$, 
\[
R^{\rho_{0}}_{(\alpha_{0},\alpha_{1},\alpha_{2})}=
\left\{
\begin{array}{ll}
\CC\langle x_{0}^{\alpha_{0}}x_{1}^{\alpha_{1}}x_{2}^{\alpha_{2}}\rangle,&\text{if}\quad \alpha_{0}\geq0\\
0,&\text{if}\quad \alpha_{0}\leq-1
\end{array}
\right.
\]
and restricting the exact sequence (\ref{Eq:Example monomial ideal}) to degree $m=(d_{1},d_{2})$, we have
\[
I^{\rho_{0}}_{m}=
\left\{
\begin{array}{ll}
\CC\langle x_{0}^{-d_{1}-d_{2}}x_{1}^{d_{1}}x_{2}^{d_{2}}\rangle\cong \bI,&\text{if}\quad -d_{1}-d_{2}\geq0\\
0,&\text{if}\quad -d_{1}-d_{2}\leq-1.
\end{array}
\right.
\]
Similarly, we obtain
\[
I^{\rho_{1}}_{m}\cong
\left\{
\begin{array}{ll}
\bI,&\text{if}\quad d_{1}\geq0\\
0,&\text{if}\quad d_{1}\leq-1
\end{array}
\right.
\quad
I^{\rho_{2}}_{m}\cong
\left\{
\begin{array}{ll}
\bI,&\text{if}\quad d_{2}\geq0\\
0,&\text{if}\quad d_{2}\leq-1.
\end{array}
\right.
\]
It only remains to compute the components in the $\Sigma-$family associated to the two dimensional cones in $\Sigma$. Let us consider $\sigma_{0}\in\Sigma(2)$ with rays $\sigma_{0}(1)=\{\rho_{1},\rho_{2}\}$. We set $R^{\sigma_{0}}:=R_{x^{\widehat{\sigma_{0}}}}$ the localization at $x^{\widehat{\sigma_{0}}}=x_{0}$ and for any multidegree $(\alpha_{0},\alpha_{1},\alpha_{2})\in\ZZ^{3}$, 
\[
R^{\sigma_{0}}_{(\alpha_{0},\alpha_{1},\alpha_{2})}=
\left\{
\begin{array}{ll}
\CC\langle x_{0}^{\alpha_{0}}x_{1}^{\alpha_{1}}x_{2}^{\alpha_{2}}\rangle,&\text{if}\quad \alpha_{1}\geq0,\;\alpha_{2}\geq0\\
0,&\text{if}\quad \alpha_{1}\leq-1\;\text{or}\;\alpha_{2}\leq -1.
\end{array}
\right.
\]
As before, taking the component of degree $m=(d_{1},d_{2})$ of (\ref{Eq:Example monomial ideal}) we obtain
\[
I^{\sigma_{0}}_{m}\cong
\left\{
\begin{array}{llll}
\bI,&\text{if}& d_{1}=0\;\text{and}\;d_{2}\geq1,\;\text{or}\;\\
&&d_{1}\geq1\;\text{and}\;d_{2}\geq0 \\
0,&& \text{otherwise}.
\end{array}
\right.
\]
Similarly, we obtain the remaining components of the $\Sigma-$family:
\[
I^{\sigma_{1}}_{m}\cong
\left\{
\begin{array}{lll}
\bI,&\text{if}& -d_{1}-d_{2}=0\;\text{and}\;d_{2}\geq2,\;\text{or}\;\\
&&-d_{1}-d_{2}\geq1\;\text{and}\;d_{2}\geq0 \\
0,&&\text{otherwise}
\end{array}
\right.
\]
\[
I^{\sigma_{2}}_{m}\cong
\left\{
\begin{array}{ll}
\bI,&\text{if}\quad -d_{1}-d_{2}\geq0\;\text{and}\;d_{1}\geq0\\
0,& \text{otherwise}.
\end{array}
\right.
\]
\end{example}
In \cite{MR-SM} the authors study in more detail the monomial ideals from the perspective of Klyachko filtrations. However, in the sequel we assume that $\cE$ is a reflexive equivariant sheaf of rank $\l$ on $X$. For any $\sigma\in\Sigma$, let $Z=\cup_{\dim\tau\geq 2} V(\tau)\subset U_{\sigma}$ be the union of the orbits of codimension $\geq 2$ with $\tau\prec\sigma$. Notice that $U_{\sigma}\setminus Z=\cup_{\rho\in\sigma(1)}U_{\rho}$ and 
\[
\Gamma(U_{\sigma},\cE)\cong
\Gamma(\bigcup_{\rho\in\sigma(1)}U_{\rho},\cE)\cong\bigcap_{\rho\in\sigma(1)}\Gamma(U_{\rho},\cE),\;\text{for any}\;\sigma\in\Sigma.
\]
Therefore, for any character $m\in M$ we have $
E^{\sigma}_{m} \cong\bigcap_{\rho\in\sigma(1)}E^{\rho}_{m}$. Thus, a $\sigma-$family is determined completely by the $\rho-$families with $\rho\in\sigma(1)$. 
Let $\rho\in\Sigma(1)$ be a ray. For any characters $m'\leq_{\rho} m$, $\chi^{\rho}_{m',m}:E^{\rho}_{m'}\rightarrow E^{\rho}_{m}$ is an isomorphism if and only if $\langle m'-m,n(\rho)\rangle=0$. Hence, for any $i\in\ZZ$ there is a unique subspace $E^{\rho}(i)\subset\CC^{\l}$ such that $E^{\rho}_{m}\cong E^{\rho}(i)$ if and only if $\langle m, n(\rho)\rangle=i$. Thus, any $\rho-$family $\hat{E}^{\rho}$ is an increasing filtration $\{E^{\rho}(i)\}_{i\in\ZZ}$ of $\CC^{\l}$ such that that $E^{\rho}(i)=0$ for $i\ll 0$ and $E^{\rho}(i)=\CC^{\l}$ for $i\gg 0$. Using this set of filtrations we have the following result:
\begin{proposition}\label{Proposition:HH0}
Let $E$ be a $(\Cl(X),\ZZ^{r})-$graded reflexive $R-$module, and $\{E^{\rho}\}_{\rho\in\Sigma(1)}$ its associated set of filtrations. Then, for any $m\in M$, 
\begin{itemize}
\item $\displaystyle{H^{0}(X, \widetilde{E})_{m} \cong \bigcap_{\rho \in \Sigma(1)} E^{\rho}(\langle m, \rho \rangle)}$.
\item $\displaystyle{H^{n}(X, \widetilde{E})_{m} \cong \bE/\sum_{\rho \in \Sigma(1)} E^{\rho}(\langle m, \rho \rangle)}$.
\item $\displaystyle{\chi(\widetilde{E})_{m} = \sum_{\sigma \in \Sigma} (-1)^{\codim \sigma} \dim_{k} E^{\sigma}_{m}}$.
\end{itemize}
\end{proposition}
\begin{proof}
See \cite[Theorem 4.1.1, Remark 4.1.2 and Corollary 4.1.3]{Kly89}, and \cite[Metatheorem 1.3.3]{Kly91}.
\end{proof}

\begin{notation}\label{Notation:Reflexives}
For each ray $\rho$, let us denote the associated filtration as $\hat{E}^{\rho}=E^{\rho}(i^{\rho}_{1},\dotsc,i^{\rho}_{\l-1},i^{\rho}_{\l};E^{\rho}_{1},\dotsc,E^{\rho}_{\l-1},E^{\rho}_{\l})$, meaning
\[
E^{\rho}(i)=
\left\{
\begin{array}{@{}l@{\hspace{1mm}}l}
0,              & i \leq i^{\rho}_{1}-1                     \\
E^{\rho}_{1},   & i^{\rho}_{1} \leq i \leq i^{\rho}_{2}-1   \\
\vdots          &                                           \\
E^{\rho}_{\l-1}, & i^{\rho}_{\l-1} \leq i \leq i^{\rho}_{\l}-1 \\
E^{\rho}_{\l},   & i^{\rho}_{\l} \leq i
\end{array}
\right.\hspace{-3mm}\;\text{where}\;
\left\{
\begin{array}{@{}l@{}}
0\subseteq E^{\rho}_{1} \subseteq \dotsb \subseteq E^{\rho}_{\l-1} \subseteq E^{\rho}_{\l}=\CC^{\l}\\
i^{\rho}_{1} \leq \dotsb \leq i^{\rho}_{\l-1} \leq i^{\rho}_{\l}\\
i^{\rho}_{j} = i^{\rho}_{j+1} \Leftrightarrow E^{\rho}_{j}=E^{\rho}_{j+1}.
\end{array}
\right.
\]
If $\Sigma(1)=\{\rho_{1},\dotsc,\rho_{r}\}$ is the (ordered) set of rays, we write $\hat{E}^{k}:=\hat{E}^{\rho_{k}}$, $i^{k}_{t}:=i^{\rho_{k}}_{t}$ and $E^{k}_{t}:=E^{\rho_{k}}_{t}$, for any $1\leq k\leq r$ and $1\leq t\leq \l$.
\end{notation}

\begin{example}
Given a Weil divisor $D=\sum_{\rho} a_{\rho} D_{\rho}$, the line bundle $\cO(D)$ corresponds to the following set of filtrations of $\CC$:
\[
O^{\rho}( i ) = 
\left\{
\begin{array}{@{}ll}
	0,	&i < -a_{\rho}	\\
	\CC,	&i \geq -a_{\rho}.
\end{array}
\right.
\]
This follows from Proposition \ref{Proposition:line bundles}, since $\cO(D)\cong\widetilde{R([\sum_{\rho}a_{\rho}D_{\rho}])}$.
\end{example}

More generally, the following result shows the effect of twisting by $\alpha\in\Cl(X)$ on the set of filtrations.

\begin{proposition}\label{Proposition:twist}
Let $E$ be a multigraded reflexive module of rank $\l$ with associated filtrations $E^{\rho}(i^{\rho}_{1},\dotsc,i^{\rho}_{\l};E^{\rho}_{1},\dotsc,E^{\rho}_{\l})$, and let $\alpha\in\Cl(X)$. The associated filtration of $E(\alpha)$ is given by $E^{\rho}(i^{\rho}_{1}-a_{\rho}, \dotsc, i^{\rho}_{\l}-a_{\rho};E^{\rho}_{1},\dotsc, E^{\rho}_{\l})$, where $D=\sum_{\rho}a_{\rho}D_{\rho}$ is any Weil divisor on $X$ such that $[D]=\alpha$. 
\end{proposition}
\begin{proof}
See \cite[Section 4.7]{Perling}.
\end{proof}

\begin{lemma}\label{Lemma:Bound reflexive filtration}
Let $E$ be a multigraded reflexive module of rank $\l$ with associated filtrations $E^{\rho}(i^{\rho}_{1},\dotsc,i^{\rho}_{\l};E^{\rho}_{1},\dotsc,E^{\rho}_{\l})$. Assume $E$ is presented as a quotient
\[
\bigoplus_{i=1}^{b_{2}}R(\um^{i})\xrightarrow{\phi}\bigoplus_{j=1}^{b_{1}}R(\umu^{j})\rightarrow E\rightarrow 0,\qquad \um^{i},\umu^{j}\in\ZZ^{r}.
\]
Then, $-\max_{j}\{\umu^{j}_{\rho}\}\leq i_{1}^{\rho}$ and $i_{\l}^{\rho}\leq -\min_{i}\{m^{i}_{\rho}\}$.
\end{lemma}
\begin{proof}
Let us consider a character $m$ and a ray $\rho\in\Sigma(1)$. If $\langle m, n(\rho)\rangle\leq \min_{j}\{-\mu^{j}_{\rho}\}-1$ then  $R(\umu^{j})^{\rho}_{m}=0$ for any $1\leq i \leq b_{1}$, and $E^{\rho}_{m}=0$. This implies that $-\max_{j}\{\mu^{j}_{\rho}\}\leq i^{\rho}_{1}$. On the other hand, if $\langle m, n(\rho)\rangle \geq \max_{i,j}\{-m^{i}_{\rho},-\mu^{j}_{\rho}\}$ then $R(\um^{i})^{\rho}_{m},\;R(\umu^{j})^{\rho}_{m}\neq0$, and  $E^{\rho}_{m}\cong \bE$. Since $\phi$ is a multiplication map by a matrix of monomials, we have that $\max_{i}\{m^{i}_{\rho}\}\leq \max_{j}\{\mu^{j}_{\rho}\}$. Therefore $\max_{i,j}\{-m^{i}_{\rho},-\mu^{j}_{\rho}\}=\max_{i}\{-m^{i}_{\rho}\}=-\min_{i}\{m^{i}_{\rho}\}$ and in particular, $i^{\rho}_{\l}\leq -\min_{i}\{m^{i}_{\rho}\}$.
\end{proof}

\begin{example}
Let $\Sigma$ be the fan of the Hirzebruch surface $\cH_{3}$ and $R=\CC[x_{0},x_{1},y_{0},y_{1}]$ be its Cox ring (see Example \ref{Example:Hirzebruch}). Then, any character $m=(d_{1},d_{2})$ is embedded in $\ZZ^{4}$ as a multidegree $m=(-d_{1}+3d_{2},d_{1},-d_{2},d_{2})$. Let $E$ be a rank $2$ reflexive module presented as:
\begin{equation}\label{Eq:Reflexive example}
R\xrightarrow{\phi}
R(0,0,1,1)
\oplus
R(0,1,0,0)
\oplus
R(1,0,0,0)
\rightarrow
E\rightarrow 0\;\text{with}\;
\phi=
\left(
\begin{array}{@{}c@{}}
3y_{0}y_{1}\\
x_{1}\\
-x_{0}
\end{array}
\right).
\end{equation}
Notice that $\widetilde{E}=\cT_{\cH_{3}}$ is the tangent sheaf on $\cH_{3}$. We want to determine $\Sigma-$family associated to $E$. First of all, we localize the multigraded sequence (\ref{Eq:Reflexive example}) at $x^{\widehat{\{0\}}}=x_{0}x_{1}y_{0}y_{1}$ and we set $R^{\{0\}}:=R_{x^{\widehat{\{0\}}}}$. Any character $m=(d_{1},d_{2})$ yields an exact sequence of vector spaces
\[
R^{\{0\}}_{m}\xrightarrow{\phi}
R^{\{0\}}(0,0,1,1)_{m}
\oplus
R^{\{0\}}(0,1,0,0)_{m}
\oplus
R^{\{0\}}(1,0,0,0)_{m}
\rightarrow
E^{\{0\}}_{m}\rightarrow 0
\]
thus we obtain $E^{\{0\}}$ as a cokernel
\[
E^{\{0\}}_{m}
\hspace{-1mm}=
\hspace{-1mm}\frac{\CC
\langle 
x_{0}^{3d_{2}-d_{1}}x_{1}^{d_{1}}y_{0}^{-d_{2}+1}y_{1}^{d_{2}+1}\hspace{-1mm},
x_{0}^{3d_{2}-d_{1}}x_{1}^{d_{1}+1}y_{0}^{-d_{2}}y_{1}^{d_{2}},
x_{0}^{3d_{2}-d_{1}+1}x_{1}^{d_{1}}y_{0}^{-d_{2}}y_{1}^{d_{2}}
\rangle
}{
\langle 
3x_{0}^{3d_{2}-d_{1}}x_{1}^{d_{1}}y_{0}^{-d_{2}+1}y_{1}^{d_{2}+1}
\hspace{-2mm}+
\hspace{-1mm}
x_{0}^{3d_{2}-d_{1}}x_{1}^{d_{1}+1}y_{0}^{-d_{2}}y_{1}^{d_{2}}
\hspace{-2mm}-
\hspace{-1mm}
x_{0}^{3d_{2}-d_{1}+1}x_{1}^{d_{1}}y_{0}^{-d_{2}}y_{1}^{d_{2}}
\rangle}
\]
together with isomorphisms
\[
\phi_{m}:E^{\{0\}}_{m}\rightarrow \frac{\CC
\langle u_{1},u_{2},u_{3}\rangle}
{
\langle
3u_{1}+u_{2}-u_{3}
\rangle
}
=\CC\langle \overline{u_{1}}, \overline{u_{2}}\rangle =:\bE\quad\text{and}\quad\overline{u_{3}}=3\overline{u_{1}}+\overline{u_{2}}.
\]
Since $E$ is a reflexive module, we only need to compute the pieces $E^{\tau}_{m}$ for each ray $\tau\in\Sigma(1)$ and $m=(d_{1},d_{2})$. Let us consider the ray $\rho_{0}\in\Sigma(1)$, and the localization at $x^{\hat{\rho_{0}}}$. We set $R^{\rho_{0}}:=R_{x^{\hat{\rho_{0}}}}$, and for any character $m=(d_{1},d_{2})$ we have
\[
R^{\rho_{0}}_{m}=
\left\{
\begin{array}{ll}
\CC
\langle
x_{0}^{3d_{2}-d_{1}}x_{1}^{d_{1}}y_{0}^{-d_{2}}y_{1}^{d_{2}}
\rangle,&
\text{if}\quad 3d_{2}-d_{1}\geq0\\
0,&\text{otherwise.}
\end{array}
\right.
\]
Therefore, localizing the exact sequence (\ref{Eq:Reflexive example}) we obtain
\[
E^{\rho_{0}}_{m}\cong
\left\{
\begin{array}{ll}
\bE,&\text{if}\quad 3d_{2}-d_{1}\geq0\\
\langle
\overline{u_{3}}
\rangle=
\langle
3\overline{u_{1}}+\overline{u_{2}}
\rangle,&
\text{if}\quad 3d_{2}-d_{1}=-1\\
0,&\text{if}\quad 3d_{2}-d_{1}\leq -2.
\end{array}
\right.
\]
Similarly, localizing with respect to the other rays, we obtain
\[
E^{\rho_{1}}_{m}\cong
\left\{
\begin{array}{ll}
\bE,&\text{if}\quad d_{1}\geq0\\
\langle
\overline{u_{2}}
\rangle,&
\text{if}\quad d_{1}=-1\\
0,&\text{if}\quad d_{1}\leq -2
\end{array}
\right.\quad
E^{\eta_{0}}_{m}\cong
\left\{
\begin{array}{ll}
\bE,&\text{if}\quad -d_{2}\geq0\\
\langle
\overline{u_{1}}
\rangle,&
\text{if}\quad -d_{2}=-1\\
0,&\text{if}\quad -d_{2}\leq -2
\end{array}
\right.
\]
\[
E^{\eta_{0}}_{m}\cong
\left\{
\begin{array}{ll}
\bE,&\text{if}\quad d_{2}\geq0\\
\langle
\overline{u_{1}}
\rangle,&
\text{if}\quad d_{2}=-1\\
0,&\text{if}\quad d_{2}\leq -2.
\end{array}
\right.
\]
Therefore, the filtrations of subspaces of $\bE\cong\CC^{2}$ associated to $E$ are
\[
\begin{array}{cc}
E^{\rho_{0}}(-1,0;\langle 3\overline{u_{1}}+\overline{u_{2}}\rangle,\bE)&
E^{\rho_{1}}(-1,0;\langle \overline{u_{2}} \rangle,\bE)\\[2mm]
E^{\eta_{0}}(-1,0;\langle \overline{u_{1}} \rangle,\bE)&
E^{\eta_{1}}(-1,0;\langle \overline{u_{1}} \rangle,\bE).
\end{array}
\]
\end{example}


\section{Multigraded regularity of reflexive multigraded modules}
\label{Section:Multigraded regularity}
In this section we focus on $(\Cl(X),\ZZ^{r})-$graded reflexive $R-$modules. We recall, in the multigraded setting, the notions of support, Castelnuovo-Mumford regularity, Hilbert function, Hilbert polynomial and regularity index. Using the $\Sigma-$family of $E$, we introduce a collection of polytopes to describe the multigraded Hilbert function of $\Gamma E$. 

 Let $E$ be a finitely generated $\Cl(X)-$graded module. We denote by $\supp(E):=\{\alpha\in\Cl(X)\;|\;E_{\alpha}\neq0\}$ the {\em multigraded support} of $E$. We define the multigraded Hilbert function to be $h_{E}:\ZZ^{c}\rightarrow \ZZ$ such that $h_{E}(\alpha)=\dim E_{\alpha}$. We fix $\cA=\{g_{1},\dotsc,g_{r}\}\subset\Cl(X)$ such that $g_{i}=\deg(x_{i})$ for $1\leq i\leq r$. For any cone $\sigma$, we set $\cA_{\hat{\sigma}}=\{g_{i}\;|\;\rho_{i}\notin\sigma\}$, and we define the subsemigroup $\cK=\bigcap_{\sigma\in\Sigma}\NN \cA_{\hat{\sigma}}\subset\NN\cA$, which is saturated because $X$ is smooth. Notice that if $\Sigma_{\max}=\{\sigma_{1},\dotsc,\sigma_{s}\}$, then $\cK=\NN\cA_{\hat{\sigma_{1}}}\cap\dotsb\cap\NN\cA_{\hat{\sigma_{s}}}$. We have the following result:

\begin{proposition}\label{Proposition:Hilbert regularity}
Let $E$ be a finitely generated $\Cl(X)-$graded $R-$module
\begin{itemize}
\item[(i)] If $\alpha\in\cK$, then $H^{i}_{B}(R)_{\alpha}=0$ for all $i\geq0$.
\item[(ii)]Assume that $E$ has a $\Cl(X)-$graded minimal free resolution
\[
0\rightarrow \bigoplus_{i}R(-\alpha_{i,t})\rightarrow
\dotsb\rightarrow \bigoplus_{i}R(-\alpha_{i,1})\rightarrow \bigoplus_{i}R(-\alpha_{i,0})\rightarrow 
E\rightarrow 0.
\]
If $\displaystyle{\beta\in\bigcap_{i,j}(\alpha_{i,j}+\cK)}$, then $H^{i}_{B}(E)_{\beta}=0$ for $i\geq0$.
\item[(iii)] There is a polynomial $P_{E}$ in $k$ variables such that $h_{E}(\alpha)-P_{E}(\alpha)=\sum_{i=0}^{d}(-1)^{i}\dim H^{i}_{B}(E)_{\alpha}$.
\end{itemize}
\end{proposition}
\begin{proof}
(i) and (ii) follow from \cite[Corollary 3.6]{Mac-Smi1}. For (iii), see \cite[Proposition 2.8 and Proposition 2.14]{Mac-Smi2}.
\end{proof}
The polynomial $P_{E}$ is called the {\em multigraded Hilbert polynomial} of $E$. In particular, by Proposition \ref{Proposition:Hilbert regularity} (ii) and the arithmetic Nullstellensatz it follows that $P_{E} = P_{\Gamma E}$. We define the {\em multigraded regularity index} of the Hilbert function of $E$ to be $\ri(E)=\{\alpha\in\Cl(X)\;|\;h_{E}(\alpha)=P_{E}(\alpha)\}$. If $\Cl(X)\cong\ZZ$, we abuse notation and write $\ri(E):=\min\{i\in \ZZ\;|\;h_{E}(i)=P_{E}(i)\}$. Hence, the definition coincides with the notion of regularity index found in \cite{Tru} or \cite{Hoa-Hyry} for the standard $\ZZ-$graded polynomial ring. 
Following \cite{Mac-Smi1} we recall the notion of Castelnuovo-Mumford multigraded regularity. We denote by $\cC=\{c_{1},\dotsc,c_{s}\}\subset\cK$ the unique minimal generating subset such that $\cK=\NN\cC$. It is called a {\em Hilbert basis} (see \cite[\textsc{iv}.16.4]{Sch}). For any integer $i\geq 0$ we set $\NN\cC[i]=\bigcup(-\lambda_{1}c_{1}-\dotsb-\lambda_{s}c_{s}+\cK)$ where the union runs over all $\lambda_{1},\dotsc,\lambda_{s}\in\NN$ such that $\lambda_{1}+\dotsb+\lambda_{s}=i$.
\begin{definition}
For $\alpha\in\Cl(X)$, the module $E$ is {\em $\alpha-$regular} if
\begin{itemize}
\item[(i)] $\underline{H^{i}_{B}(E)_{\beta}=0}$ for any $i\geq1$ and all $\underline{\beta\in \alpha+\NN\cC[i-1]}$, and
\item[(ii)] $\underline{H^{0}_{B}(E)_{\beta}=0}$ for all $\beta\in\bigcup_{i=1}^{s}(\alpha+c_{i}+\cK)$.
\end{itemize}
The set $\reg(E)=\{\alpha\in\Cl(X)\;|\;E\;\text{is}\;\alpha-\text{regular}\}$ is called the {\em multigraded Castelnuovo-Mumford regularity} of $E$.
\end{definition}

Let $\cA\subset \ZZ^{c}$ be a set. We say that $B\subset \ZZ^{c}$ is a {\em lower bound} (respectively {\em upper bound}) of $\cA$ if $\cA\subset B$ (respectively $B\subset \cA$). For instance, from Proposition \ref{Proposition:Hilbert regularity} (ii) and (iii) the set $\bigcap_{i,j}(\alpha_{i,j}+\cK)$ is a lower bound of $\ri(E)$. On the other hand, we define $\reg(E)^{\circ}=\bigcup_{\alpha\in\reg(E)}((\alpha+\cK)\setminus\{\alpha\})\subset \reg(E)$. The following result relates the multigraded Castelnuovo-Mumford regularity with the multigraded regularity index:
\begin{proposition}
If $E$ is $\alpha-$regular, then $(\alpha+\cK)\setminus\{\alpha\}\subset \ri(E)$. In particular $\reg(E)^{\circ}\subset \ri(E)$, and thus $\ri(E)$ is a lower bound of $\reg(E)^{\circ}$.
\end{proposition}
\begin{proof}
See \cite[Corollary 2.15]{Mac-Smi2}.
\end{proof}

Next, we study these invariants in the case of reflexive $(\Cl(X),\ZZ^{r})-$graded $R-$modules, where $X$ is a complete smooth toric variety of dimension $n$, $|\Sigma|=\RR^{n}$ and $\Sigma_{\max}=\{\sigma_{1},\dotsc,\sigma_{s}\}$. We choose a basis $\{e_{1},\dotsc,e_{n}\}$ of $N=\ZZ^{n}$ such that $\sigma_{1}=\cone(e_{1},\dotsc,e_{n})$. We write $\Sigma(1)=\{\rho_{1},\dotsc,\rho_{n},\rho_{n+1},\dotsc,\rho_{r}\}$ the set of rays such that $\rho_{k}=\cone(e_{k})$ for $1\leq k\leq n$. The class group of $X$ is $\Cl(X)=\coker(\phi)\cong\ZZ^{r-n}$ with $\phi:M\rightarrow\ZZ^{r}$ given by the following matrix
\begin{equation}\label{Eq:Notation smooth toric class group}
\left(
\begin{array}{@{}ccc@{}}
1&\cdots&0\\
\vdots&\ddots&\vdots\\
0&\cdots&1\\
a_{1}^{1}&\cdots&a_{n}^{1}\\
\vdots&&\vdots\\
a_{1}^{r-n}&\cdots&a_{n}^{r-n}
\end{array}
\right)\;\text{where}\;
\rho_{k}=\cone(a_{1}^{k}e_{1}+\dotsb+a_{n}^{k}e_{n}).
\end{equation}
Then, $[D_{\rho_{k}}]=a_{k}^{1}[D_{\rho_{n+1}}]+\dotsb+a_{k}^{r-n}[D_{\rho_{r}}]$ for $1\leq k\leq n$. Hence $\Cl(X)=\ZZ\langle[D_{\rho_{n+1}}],\dotsc,[D_{\rho_{r}}]\rangle$. The Cox ring of $X$ is $R=\CC[x_{1},\dotsc,x_{r}]$ with $\deg(x_{k})=(a_{k}^{1},\dotsc,a_{k}^{r-n})$ for $1\leq k\leq n$ and $\deg(x_{n+k})=\alpha_{k}$ for $1\leq k\leq r-n$, where $\{\alpha_{1},\dotsc,\alpha_{r-n}\}$ is the standard basis of $\ZZ^{r-n}$. As in Subsection \ref{Section:Preliminaries Klyachko filtrations for modules}, we identify a character $m=(d_{1},\dotsc,d_{n})\in M$ with its image by $\phi$:
\[
\phi(m)=(d_{1},\dotsc,d_{n},a_{1}^{1}d_{1}+\dotsb+a_{n}^{1}d_{n},\dotsc,
a_{1}^{r-n}d_{1}+\dotsb+a_{n}^{r-n}d_{n}).
\]
Let $E$ be a $(\Cl(X),\ZZ^{r})-$graded reflexive module of rank $\l$ and $\{\hat{E}^{k}\}_{k=1}^{r}$ the associated $\Sigma-$family, where $\hat{E}^{k}=E^{k}(i_{1}^{k},\dotsc,$ $i_{\l}^{k};E_{1}^{k},\dotsc,E_{\l-1}^{k},\CC^{\l})$. For any $\um\in \{1,\dotsc,\l\}^{r}$, and $\ut=(t_{1},\dotsc,t_{r-n})\in\Cl(X)$ we consider the $n-$dimensional polytope $\Omega_{\um}(\ut)$, defined by linear system of inequalities
\begin{equation}\label{Eq:Polytop system reflexive}
\left(
\begin{array}{@{}c@{}}
i_{m_{1}}^{1}\\
\vdots\\
i_{m_{n}}^{n}\\
i_{m_{n+1}}^{n+1}-t_{1}\\
\vdots\\
i_{m_{r}}^{r}-t_{r}
\end{array}
\right)
\leq\left(
\begin{array}{@{}ccc@{}}
1&\cdots&0\\
\vdots&\ddots&\vdots\\
0&\cdots&1\\
a_{1}^{1}&\cdots&a_{n}^{1}\\
\vdots&&\vdots\\
a_{1}^{r-n}&\cdots&a_{n}^{r-n}
\end{array}
\right)\cdot m\leq
\left(
\begin{array}{@{}c@{}}
i_{m_{1}+1}^{1}\\
\vdots\\
i_{m_{n}+1}^{n}\\
i_{m_{n+1}+1}^{n+1}-t_{1}\\
\vdots\\
i_{m_{r}+1}^{r}-t_{r}
\end{array}
\right)
\end{equation}
where we set $i^{k}_{\l+1}:=\infty$. We define $\Psi_{\um}(\ut)=\Omega_{\um}(\ut)\cap\ZZ^{n}$, the integer solutions of (\ref{Eq:Polytop system reflexive}). We have the following result.
\begin{proposition}\label{Proposition:Hilbert from polytope}
Let $\ut\in\Cl(X)$.
\begin{itemize}
\item[(i)] $\um\in\{1,\dotsc,\l\}^{r}$ and $m\in\Psi_{\um}(\ut)$ if and only if 
\[
[\Gamma E(\ut)]_{m}=\bigcap_{k=1}^{r}E^{k}_{m_{k}}.
\]
\item[(ii)] If $m\in M$ is a character such that $[\Gamma E(\ut)]_{m}\neq0$, then there is $\um\in\{1,\dotsc,\l\}^{r}$ such that $m\in\Psi_{\um}(\ut)$.
\item[(iii)] The Hilbert function of $\Gamma E$ is
\[
h_{\Gamma E}(\ut)=\sum_{\um\in\{1,\dotsc,\l\}^{r}}|\Psi_{\um}(\ut)|D(\um)
\]
where for $\um=(m_{1},\dotsc,m_{r})$, $D(\um)=\dim\bigcap_{k=1}^{r}
E^{k}_{m_{k}}$.
\end{itemize}
\end{proposition}
\begin{proof}
Notice that $D=t_{1}D_{\rho_{n+1}}+\dotsb+t_{r-n}D_{\rho_{r}}$ is a Weil divisor with $[D]=\ut\in\Cl(X)$. Hence, from Propositions \ref{Proposition:line bundles} and \ref{Proposition:twist}, the filtration of $E(\ut)$ is
\vspace{-4mm}
\[
\left\{
\begin{array}{@{}ll}
E^{k}(i_{1}^{k},\dotsc,i_{\l}^{k};E^{k}_{1},\dotsc,\CC^{\l})&1\leq k\leq n\\
E^{k}(i_{1}^{k}-t_{k},\dotsc,i_{\l}^{k}-t_{k};E^{k}_{1},\dotsc,\CC^{\l})&n+1\leq k\leq r.
\end{array}\right.
\]
On the other hand, $m$ is a solution of (\ref{Eq:Polytop system reflexive}) if and only if $i_{m_{k}}^{k}\leq \langle m,n(\rho_{k})\rangle <i_{m_{k}+1}^{k}$ for $1\leq k\leq n$ and
$i_{m_{k}}^{n+k}-t_{k}\leq \langle m, n(\rho_{n+k})\rangle <
i_{m_{n+k}+1}^{n+k}-t_{k}$. So, the result follows by Proposition \ref{Proposition:HH0} and the interpretation of Notation \ref{Notation:Reflexives}.
\end{proof}

Proposition \ref{Proposition:Hilbert from polytope} provides bounds for the the multigraded support of $\Gamma E$ and it encodes a way to compute the multigraded Hilbert function. In particular, studying asymptotically $|\Psi_{\um}(\ut)|$ with respect to $\ut$ we can see when it behaves as a polynomial in $\ut$, and gather information on the multigraded regularity index $\ri(\Gamma E)$. 
However, it is worth noticing that the lattice polytopes $\Omega_{\um}(\ut)$ can be complicated, and their associated sets of integer points $\Psi_{\um}(\ut)$ can be difficult to control, in general. Nonetheless, in next section we see that for toric varieties with $1-$splitting fan we can find explicit bounds for the multigraded regularity index and a closed expression for the multigraded Hilbert polynomial. 


\section{Application to Cox rings with splitting fans}
\label{Section:Applications}
In this section, we apply Proposition \ref{Proposition:Hilbert from polytope} to reflexive modules over the ring $R=\CC[x_{0},\dotsc,x_{s},y_{0},\dotsc,y_{r}]$ $\ZZ^{2}-$graded by $\deg(x_{i})=(1,0)$ for $0\leq i \leq s$, $\deg(y_{0})=(0,1)$ and $\deg(y_{i})=(-a_{i},1)$ for $1\leq i\leq r$. 
When $a_{i}=0$ for $1\leq i\leq r$, $R$ is the classical bihomogeneous polynomial ring. 
In general, $R$ is the Cox ring of a toric variety with $1-$splitting fan. As we recall next, the fan of this toric variety is a {\em splitting fan}. More precisely, using the Batyrev notation we have the following definition.
\begin{definition}\label{Definition:Splitting}
Let $\Sigma$ be the fan of a $d-$dimensional smooth complete toric variety. A set of rays $\{\rho_{1},\dotsc,\rho_{k}\}$ is called a {\em primitive collection} if $\cone(n(\rho_{1}),\dotsc,n(\rho_{k}))\notin\Sigma$, but $\cone(n(\rho_{1}),\dotsc,\widehat{n(\rho_{j})},\dotsc,$ $n(\rho_{k}))\in\Sigma$. If $\{\rho_{1},\dotsc,\rho_{k}\}$ is a primitive collection, then there is a unique cone $\sigma=\cone(v_{1},\dotsc,v_{t})\in\Sigma$ and integers $a_{1},\dotsc,a_{t}>0$ such that $n(\rho_{1})+\dotsb+n(\rho_{k})=a_{1}v_{1}+\dotsb+a_{t}v_{t}$, which is called the associated {\em primitive relation}. We say that $\Sigma$ is a {\em splitting fan} if any two primitive collections have no common elements.
\end{definition} 
\begin{proposition}
Let $\Sigma$ be the fan of a $d-$dimensional smooth complete toric variety. The fan $\Sigma$ is a splitting fan if and only if there is a sequence of toric varieties $X=X_{k},\dotsc,X_{0}$ such that $X_{0}=\PP^{n}$ for some $n$ and for $1\leq i\leq k$, $X_{i}$ is a projectivization of a decomposable vector bundle over $X_{i-1}$. Moreover, we say then that the fan is $k-$splitting.

\end{proposition}
\begin{proof}
See \cite[Theorem 4.3 and Corollary 4.4]{Bat}.
\end{proof}

In particular, a toric variety has $1-$splitting fan if and only if $X$ is a $(r+s)-$dimensional toric variety $V_{s}(a_{1},\dotsc,a_{r})=\PP(\cO_{\PP^{s}}\oplus\cO_{\PP^{s}}(a_{1})\oplus\dotsb\oplus\cO_{\PP^{s}}(a_{r}))$. Let us introduce the fan $\Sigma\subset\RR^{r+s}$ of $V_{s}(a_{1},\dotsc,a_{r})$. We fix $\{e_{1},\dotsc,e_{s},f_{1},\dotsc,f_{r}\}$ the standard basis for $\ZZ^{s+r}$. We define 
\[
\begin{array}{l@{}ll@{}l}
\rho_{0}&:=\cone(-e_{1}-\dotsb-e_{s}+a_{1}f_{1}+\dotsb+a_{r}f_{r})&
\eta_{0}&:=\cone(-f_{0}-\dotsb-f_{r})\\
\rho_{i}&:=\cone(e_{i})\quad 1\leq i\leq s &
\eta_{j}&:=\cone(f_{j})\quad 1\leq j\leq r,
\end{array}
\]
and for $1\leq i\leq s$ and $1\leq j\leq r$ we define the $r+s-$dimensional cones
\[
\sigma_{ij}:=\cone(n(\rho_{0}),\dotsc,\widehat{n(\rho_{i})},\dotsc,n(\rho_{s}),n(\eta_{0}),\dotsc,\widehat{n(\eta_{j})},\dotsc,n(\eta_{r})).
\] 
Then $\Sigma(1)\hspace{-1mm}=\hspace{-1mm}\{\rho_{0},\dotsc,\rho_{s},\eta_{0},\dotsc,\eta_{r}\}$, and
$\Sigma_{max}\hspace{-1mm}=\hspace{-1mm}\{\sigma_{ij}\,|\,1\leq i\leq s,\,1\leq j\leq r\}$. 
From the exact sequence (\ref{Eq:Class group exact sequence}) we obtain the class group of $X=V_{s}(a_{1},\dotsc,a_{r})$ as $\Cl(X)=\coker\phi$, with $\phi:\ZZ^{s+r}\rightarrow\ZZ^{s+r+2}$ as in (\ref{Eq:Notation smooth toric class group}). Therefore, we have $\Cl(X)=\ZZ\langle[D_{\rho_{0}}],[D_{\eta_{0}}]\rangle$, $[D_{\rho_{i}}]=[D_{\rho_{0}}]$ for $1\leq i\leq s$ and $[D_{\eta_{j}}]=-a_{j}[D_{\rho_{0}}]+[D_{\eta_{0}}]$ for $1\leq j\leq r$.

\begin{notation}\label{Notation:Reflexives on V(s,a1,...,ar)}
Let $E$ be a $(\Cl(X),\ZZ^{s+r+2})-$graded reflexive $R-$module of rank $\l$. Let $\hat{E}=\{\hat{E}^{\tau}=:E^{\tau}(k^{\tau}_{1},\dotsc,k^{\tau}_{\l};E^{\tau}_{1},\dotsc,E^{\tau}_{\l})\}_{\tau\in\Sigma(1)}$ be its associated $\Sigma-$family (see Notation \ref{Notation:Reflexives}). We rename these filtrations as follows:
\[\text{for}\; 1\leq n\leq \l,\;
\left\{
\begin{array}{@{}ll}
i^{t}_{n}:=k^{\rho_{t}}_{n},&0\leq t\leq r\\
j^{u}_{n}:=k^{\eta_{u}}_{n},&0\leq u\leq s
\end{array}
\right.\;\text{and}\;
\left\{
\begin{array}{@{}ll}
F^{t}_{n}:=E^{\rho_{t}}_{n},&0\leq t\leq r\\
G^{u}_{n}:=E^{\eta_{u}}_{n},&0\leq u\leq s.
\end{array}
\right.
\]
In particular, we set 
\[
\delta_{E}:=\left(\sum_{t=0}^{s}i^{t}_{\l}-\sum_{u=1}^{s}a_{u}j^{u}_{\l},\sum_{u=0}^{s}j^{u}_{\l}\right)=\sum_{\tau}k^{\tau}_{\l}[D_{\tau}]\in\Cl(X)
.\]
By Proposition \ref{Proposition:twist}, the reflexive module $E(\delta_{E})=E\otimes R(\delta_{E})$ corresponds to the filtration  $\{E^{\tau}(k^{\tau}_{1}-k^{\tau}_{\l},\dotsc,k^{\tau}_{\l-1}-k^{\tau}_{\l},0;E^{\tau}_{1},\dotsc,E^{\tau}_{\l})\}_{\tau\in\Sigma(1)}$. Finally, given $\alpha=(p,q)=p[D_{\rho_{0}}]+q[D_{\eta_{0}}]$, we denote by $\hat{E}(p,q)$ the $\Sigma-$family of the reflexive module $E(p,q)$. By Proposition \ref{Proposition:twist}, 
\[
\hat{E}(p,q)^{\tau}\!\!=\!\hat{E}^{\tau}\;\text{for}\;\tau\in\Sigma(1)\setminus\{\rho_{0},\eta_{0}\}\,\text{and}\,
\left\{
\begin{array}{@{}l@{}}
\hat{E}(p,q)^{\rho_{0}}(k)=\hat{E}^{\rho_{0}}(k+p)\\
\hat{E}(p,q)^{\eta_{0}}(k)=\hat{E}^{\eta{0}}(k+q)
\end{array},
\,k\in\ZZ.
\right.
\]
\end{notation}

Our first goal is to bound $\supp(\Gamma E)$. Using Propositions \ref{Proposition:HH0} and the above twisted filtrations, for any $(p,q)\in\ZZ^{2}$ and $m\in M$, we have:
\begin{multline}\label{Decomposition E_(t,s)}
[\Gamma E_{(p,q)}]_{m}\cong H^{0}(X,\tilde{E}(p,q))_{m}\cong
\hat{E}^{\rho_{0}}(\langle m, \rho_{0}\rangle+p)\cap \hat{E}^{\rho_{1}}(\langle m, \rho_{1}\rangle)\cap\dotsb\\
\cap\hat{E}^{\rho_{s}}(\langle m, \rho_{s}\rangle)\cap 
\hat{E}^{\eta_{0}}(\langle m, \eta_{0}\rangle+q)\cap \hat{E}^{\eta_{1}}(\langle m, \eta_{1}\rangle)\cap\dotsb\cap\hat{E}^{\eta_{r}}(\langle m, \eta_{r}\rangle).
\end{multline} 
The following lemmas will be very useful.

\begin{lemma}\label{Lemma:System1}
Let $0\leq a_{1}\leq \dotsb \leq a_{r}$, $A$ and $B\geq 0$ be integers. The system
\begin{equation}\label{System1}
\left.
\begin{array}{r@{}l@{}}
a_{1} x_{1}+\dotsb+ a_{r} x_{r}\geq A\\
x_{1}+\dotsb+x_{r}\leq B
\end{array}
\right\}
\end{equation}
has a solution in $\ZZ^r_{\geq0}$ if and only if $A\leq a_{r} B$.
\end{lemma}
\begin{proof}
First, we assume that $(x_{1},\dotsc,x_{r})\in\ZZ^{r}_{\geq0}$ is a solution of (\ref{System1}). In particular,
$x_{1}+\dotsb+x_{r}\leq B$ and $A\leq a_{1}x_{1}+\dotsb+a_{r}x_{r}$. Now, since $0\leq a_{1}\leq\dotsb \leq a_{r}$, we get $A\leq a_{1}x_{1}+\dotsb+a_{r}x_{r}\leq a_{r}x_{1}+\dotsb+a_{r}x_{r}\leq a_{r}B$, yielding $A\leq a_{r}B$. Conversely, if $A\leq a_{r} B$ then $(0,\dotsc,0,B)\in\ZZ^{r}_{\geq0}$ is directly a solution of (\ref{System1}).
\end{proof}

\begin{lemma}\label{Lemma:System2}
Let $0\leq a_{1}\leq \dotsb \leq a_{r}$, $A$ and $B$ be integers. The system
\begin{equation}\label{System2}
\left(
\begin{array}{@{}cccccc@{}}
1&\cdots&1&-a_{1}&\cdots&-a_{r}\\
0&\cdots&0&1&\cdots&1
\end{array}
\right)\cdot\ue\leq
\left(
\begin{array}{@{}r@{}}
-A\\
B
\end{array}
\right)
\end{equation}

has a solution $\ue\in\ZZ^{s+r}_{\geq0}$ if and only if $B\geq0$ and the system
\begin{equation}\label{System3}
\left.
\begin{array}{r@{}l@{}}
a_{1} x_{1}+\dotsb+ a_{r} x_{r}\geq &A\\
x_{1}+\dotsb+x_{r}\leq &B
\end{array}
\right\}
\end{equation}
 has a solution in $\ZZ^{r}_{\geq0}$.
\end{lemma}
\begin{proof}
First we assume that $(e_{1},\dotsc,e_{s+r})\in\ZZ^{s+r}_{\geq0}$ is a solution of (\ref{System2}). Then $0\leq e_{s+1}+\dotsb+e_{s+r}\leq B$, and hence $B\geq0$. On the other hand, $a_{1}e_{s+1}+\dotsb+a_{r}e_{s+r}\geq e_{1}+\dotsb+e_{s}+A\geq A$. Thus, $(e_{s+1},\dotsc,e_{s+r})\in\ZZ^{r}_{\geq0}$ is a solution of (\ref{System3}). Conversely, if $B\geq0$ and $(x_{1},\dotsc,x_{r})$ is a solution of (\ref{System3}), then $(0,\dotsc,0,x_{1},\dotsc,x_{r})$ is directly a solution of (\ref{System2}).
\end{proof}

\begin{lemma}\label{Lemma:Metasystem}
Let $0\leq a_{1}\leq \dotsb \leq a_{r}$, $\lambda_{0},\dotsc,\lambda_{s},\mu_{0},\dotsc,\mu_{r}$ be integers. The system
\begin{equation}\label{Metasystem}
\left(
\begin{array}{@{}ccc@{\hspace{2mm}}ccc@{}}
-1&\cdots&-1&a_{1}&\cdots&a_{r}\\
1&\cdots&0&0&\cdots&0\\
\vdots&\ddots&\vdots&\vdots&&\vdots\\
0&\cdots&1&0&\cdots&0\\
0&\cdots&0&-1&\cdots&-1\\
0&\cdots&0&1&\cdots&0\\
\vdots&&\vdots&\vdots&\ddots&\vdots\\
0&\cdots&0&0&\cdots&1
\end{array}
\right)\cdot m\geq
\left(
\begin{array}{@{}c@{}}
\lambda_{0}\\
\lambda_{1}\\
\vdots\\
\lambda_{s}\\
\mu_{0}\\
\mu_{1}\\
\vdots\\
\mu_{r}
\end{array}
\right)
\end{equation}
has a solution $m\in\ZZ^{s+r}$ if and only if the system
\begin{equation}
\left(
\begin{array}{@{}cccccc@{}}
1&\cdots&1&-a_{1}&\cdots&-a_{r}\\
0&\cdots&0&1&\cdots&1
\end{array}
\right)\cdot\ue\leq
\left(
\begin{array}{@{}r@{}}
-\lambda_{0}-
\hspace{-1mm}\dotsb\hspace{-1mm}
-\lambda_{s}+a_{1}\mu_{1}+
\hspace{-1mm}\dotsb\hspace{-1mm}
+a_{r}\mu_{r}\\
-\mu_{0}-
\hspace{-1mm}\dotsb\hspace{-1mm}
-\mu_{r}
\end{array}
\right)
\end{equation}
has solutions in $\ZZ^{s+r}_{\geq0}$.
\end{lemma}
\begin{proof}
It follows using the change of variables $e_{1}=d_{1}-\lambda_{1},\dotsc,e_{s}=d_{s}-\lambda_{s},\,e_{s+1}=d_{s+1}-\mu_{1},\dotsc,e_{s+r}=d_{s+r}-\mu_{r}$, where $m=(d_{1},\dotsc,d_{s+r})$.
\end{proof}

\begin{proposition}
The set
\[
L_{E}:=
\left\{\hspace{-1mm}
(p,q)\hspace{-1mm}
\left|
\begin{array}{@{}l@{}l@{}l@{}}
q\geq j_{1}^{0}+	&\dotsb+j_{1}^{r}\\
p+a_{r}q 			&\geq i_{1}^{0}+\dotsb+i_{1}^{s}+a_{r}j_{1}^{0}+\\&(a_{r}-a_{1})j_{1}^{1}+\dotsb+(a_{r}-a_{r-1})j_{1}^{r-1}
\end{array}
\right.\hspace{-1mm}
\right\}
\]
is a lower bound of $\supp(\Gamma E)$.
\end{proposition}
\begin{proof}
We see that $\supp(\Gamma E)\subset L_{E}$. Let $(p,q)\in \ZZ^{2}$ be such that $\Gamma E_{(p,q)}\neq0$, and $m\in\ZZ^{s+r}$ with $[\Gamma E_{(p,q)}]_{m}\neq0$. By (\ref{Decomposition E_(t,s)}) $\dim\hat{E}(p,q)^{\tau}(\langle m,\tau\rangle)\geq 1$. Hence, $m$ is a solution of (\ref{Metasystem}) setting $\lambda_{0}=i_{1}^{0}-p,\,\lambda_{1}=i_{1}^{1},\dotsc,\lambda_{s}=i_{1}^{s},\,\mu_{0}=j_{1}^{0}-q,\,\mu_{1}= j_{1}^{1},\dotsc,\mu_{r}=j_{1}^{r}$.  Then, by Lemma \ref{Lemma:Metasystem} the system 
\[\left.
\begin{array}{r@{}}
e_{1}+
\hspace{-1mm}\dotsb\hspace{-1mm}
+e_{s}-a_{1} e_{s+1}+
\hspace{-1mm}\dotsb\hspace{-1mm}
+ a_{r} e_{s+r}\leq i_{1}^{0}+
\hspace{-1mm}\dotsb\hspace{-1mm}
+i_{1}^{s}-p-a_{1}j_{1}^{1}-
\hspace{-1mm}\dotsb\hspace{-1mm}
-a_{r}j_{1}^{r}\\
e_{s+1}+
\hspace{-1mm}\dotsb\hspace{-1mm}
+e_{s+r}\leq q-j_{1}^{0}-
\hspace{-1mm}\dotsb\hspace{-1mm}
-j_{1}^{r}
\end{array}
\right\}
\]
has a solution in $\ZZ^{s+r}_{\geq0}$. Finally, by Lemmas \ref{Lemma:System2} and \ref{Lemma:System1}, this implies that $q-j_{1}^{0}-\dotsb-j_{1}^{r}\geq0$ and $i_{1}^{0}+\dotsb+i_{1}^{s}-p-a_{1}j_{1}^{1}-\dotsb-a_{r}j_{1}^{r}\leq a_{r}(q-j_{1}^{0}-\dotsb-j_{1}^{r})$. Thus, $q\geq j_{1}^{0}+\dotsb+j_{1}^{r}$ and $p+a_{r}q \geq i_{1}^{0}+\dotsb+i_{1}^{s}+a_{r}j_{1}^{0}+(a_{r}-a_{1})j_{1}^{1}+\dotsb+(a_{r}-a_{r-1})j_{1}^{r-1}$, and the result follows.
\end{proof}

Next we define the following subsets of $\ZZ^{2}$:
\[
I(k):=
\left\{\hspace{-1mm}
(p,q)\hspace{-1mm}
\left|
\begin{array}{@{}l@{\,}l@{}}
q\;\geq\; j_{\l}^{0}+&\dotsb+j_{\l}^{r}\\
p+a_{r}q\geq& i_{\l}^{0}+\dotsb+ i_{1}^{k}+\dotsb+i_{\l}^{s} +a_{r}j_{\l}^{0}+\\
& (a_{r}-a_{1})j_{\l}^{1}+\dotsb+
(a_{r}-a_{r-1})j_{\l}^{r-1}
\end{array}
\right.
\hspace{-1mm}
\right\}\!,0\leq k\leq s
\]
\[
J(k)\!:=\!
\left\{\hspace{-1mm}
(p,q)\hspace{-1mm}
\left|
\begin{array}{@{}l@{\,}l@{\;}}
q\;\geq \;j_{\l}^{0}+&\dotsb+j_{1}^{k}+\dotsb+j_{\l}^{r}\\
p+a_{r}q\geq &i_{\l}^{0}+\dotsb+i_{\l}^{s}+a_{r}j_{\l}^{0}+(a_{r}-a_{1})j_{\l}^{1}+\\
&\dotsb+(a_{r}-a_{k})j_{1}^{k}\dotsb+(a_{r}-a_{r-1})j_{\l}^{r-1}
\end{array}
\right.
\hspace{-3mm}
\right\}\!,0\leq k\leq r.
\]
\begin{proposition}
If either $(p,q)\in I(k)$ for any $0\leq k\leq s$, or $(p,q)\in J(k)$ for any $0\leq k\leq r$, then $\Gamma E_{(p,q)}\neq0$. In particular,
\[
U_{E}:=\left(\bigcup_{k=0}^{s}I(k)\right)\cup\left(\bigcup_{k=0}^{r}J(k)\right)
\]
is an upper bound of $\supp(\Gamma E)$.
\end{proposition}
\begin{proof}
If $(p,q)\in I(k)$ for some integer $0\leq k\leq s$, we have that $q- j_{\l}^{0}-\dotsb-j_{\l}^{r}\geq0$ and $i_{\l}^{0}+\dotsb+i_{1}^{k} +\dotsb+i_{\l}^{s}-p-a_{1}j_{\l}^{1}-\dotsb-a_{r}j_{\l}^{r}\leq a_{r}(q-j_{\l}^{0}-\dotsb-j_{\l}^{r})$. Hence, by Lemmas \ref{Lemma:System1}, \ref{Lemma:System2} and \ref{Lemma:Metasystem}, the system of inequalities
\[
\left(
\begin{array}{@{}ccc@{\hspace{2mm}}ccc@{}}
-1&\cdots&-1&a_{1}&\cdots&a_{r}\\
1&\cdots&0&0&\cdots&0\\
\vdots&\ddots&\vdots&\vdots&&\vdots\\
0&\cdots&1&0&\cdots&0\\
0&\cdots&0&-1&\cdots&-1\\
0&\cdots&0&1&\cdots&0\\
\vdots&&\vdots&\vdots&\ddots&\vdots\\
0&\cdots&0&0&\cdots&1
\end{array}
\right)\cdot m\geq
\left(
\begin{array}{@{}c@{}}
i_{\l}^{0}-p\\
i_{\l}^{1} \\
\vdots\\
i_{1}^{k}\\
\vdots\\
i_{\l}^{s}\\
j_{\l}^{0}-q\\
j_{l}^{1}\\
\vdots\\
j_{\l}^{r}
\end{array}
\right)
\]
has a solution $m\in\ZZ^{s+r}$. By Notations \ref{Notation:Reflexives} and \ref{Notation:Reflexives on V(s,a1,...,ar)}, this means that $\hat{E}(p,q)(\langle m, \tau\rangle)\cong\CC^{\l}$ for all $\tau\in\Sigma(1)\setminus\{\rho_{k}\}$ and $\dim\hat{E}(p,q)(\langle m, \rho_{k}\rangle)\geq1$. Thus, from (\ref{Decomposition E_(t,s)}), it follows that $\dim[\Gamma E_{(p,q)}]_{m}\geq1$ and hence $\Gamma E_{(p,q)}\neq0$. Symmetrically, if $(p,q)\in J(k)$ for some integer $0\leq k\leq r$, we obtain that  $\hat{E}(p,q)(\langle m, \tau\rangle)\cong\CC^{\l}$ for all $\tau\in\Sigma(1)\setminus\{\eta_{k}\}$ and $\dim\hat{E}(p,q)(\langle m, \eta_{k}\rangle)\geq1$, implying that $\Gamma E_{(p,q)}\neq0$.
\end{proof}

The end of the section is devoted to find a bound for the multigraded regularity index $\ri(\Gamma E)$ using Proposition \ref{Proposition:Hilbert from polytope} (iii). Let $\um\in\{1,\dotsc,\l\}^{s}$ and $\un\in\{1,\dotsc,\l\}^{r}$ and we write $\um:=(m_{0},\dotsc,m_{s})$ and $\un:=(n_{0},\dotsc,n_{r})$. The polytope $\Omega_{(\um,\un)}(p,q)$ defined in (\ref{Eq:Polytop system reflexive}) is given by the following system of linear inequalities:
\begin{equation}\label{System intersection}
\left(
\begin{array}{@{}l@{}}
i_{m_{0}+1}^{0}-p\\
i_{m_{1}+1}^{1} \\
\vdots\\
i_{m_{s}+1}^{s}\\
j_{n_{0}+1}^{0}-q\\
j_{n_{1}+1}^{1}\\
\vdots\\
j_{n_{r}+1}^{r}
\end{array}
\right)>
\left(
\begin{array}{@{}ccc@{\hspace{2mm}}ccc@{}}
-1&\cdots&-1&a_{1}&\cdots&a_{r}\\
1&\cdots&0&0&\cdots&0\\
\vdots&\ddots&\vdots&\vdots&&\vdots\\
0&\cdots&1&0&\cdots&0\\
0&\cdots&0&-1&\cdots&-1\\
0&\cdots&0&1&\cdots&0\\
\vdots&&\vdots&\vdots&\ddots&\vdots\\
0&\cdots&0&0&\cdots&1
\end{array}
\right)\cdot m\geq
\left(
\begin{array}{@{}l@{}}
i_{m_{0}}^{0}-p\\
i_{m_{1}}^{1} \\
\vdots\\
i_{m_{s}}^{s}\\
j_{n_{0}}^{0}-q\\
j_{n_{1}}^{1}\\
\vdots\\
j_{n_{r}}^{r}
\end{array}
\right)
\end{equation}
as before, we set $i^{k}_{\l+1}=j^{k}_{\l+1}:=\infty$ and $\Psi_{(\um,\un)}(p,q):=\Omega_{(\um,\un)}(p,q)\cap\ZZ^{s+r}$.
 
\begin{remark}\label{Remark:HilbertFunction}
Recall that 
$m\in\Psi_{(\um,\un)}(p,q)$, if and only if $[\Gamma E_{(p,q)}]_{m} = F_{m_{0}}^{0}\cap F_{m_{1}}^{1}\cap\dotsb\cap F_{m_{s}}^{s}\cap G_{n_{0}}^{0}\cap G_{n_{1}}^{1}\cap\dotsb \cap G_{n_{r}}^{r}$. 
In particular, if $m\in\ZZ^{s+r}$ and $[\Gamma E_{(p,q)}]_{m}\neq0$, then $m\in \Psi_{(\um,\un)}(p,q)$ for some $\um$ and $\un$.
\end{remark}
To study the behaviour of $\Psi_{(\um,\un)}(p,q)$ with respect to $(p,q)$ from the system (\ref{System intersection}), we introduce two smaller dimensional systems. Let us first define $\Psi_{\un}(q)\subset \ZZ^{r}$ as the set of integer solutions of the system:
\begin{equation}\label{System jotes}
\left(
\begin{array}{@{}l@{}}
j_{n_{0}+1}^{0}\\
j_{n_{1}+1}^{1}\\
\vdots\\
j_{n_{r}+1}^{r}
\end{array}
\right)>
\left(
\begin{array}{@{}ccc@{}}
-1&\cdots&-1\\
1&\cdots &0\\
\vdots&\ddots&\vdots\\
0&\cdots&1
\end{array}
\right)\cdot \uc\geq
\left(
\begin{array}{@{}l@{}}
j_{n_{0}}^{0}-q\\
j_{n_{1}}^{1}\\
\vdots\\
j_{n_{r}}^{r}
\end{array}
\right).
\end{equation}
For any $\uc:=(c_{1},\dotsc,c_{r})\in \Psi_{\un}(q)$ we set $A(\uc):=a_{1} c_{1}+\dotsb+a_{r}c_{r}$ and we define $\Psi_{\um}(p;\uc)\subset\ZZ^{s}$ to be the set of integer solutions of the system:

\begin{equation}\label{System i's}
\left(
\begin{array}{l}
i_{m_{0}+1}^{0}-p-A(\uc)\\
i_{m_{1}+1}^{0}\\
\vdots\\
i_{m_{s}+1}^{0}
\end{array}
\right)>
\left(
\begin{array}{@{}ccc@{}}
-1&\cdots&-1\\
1&\cdots&0\\
\vdots&\ddots&\vdots\\
0&\cdots&1
\end{array}
\right)\cdot m\geq
\left(
\begin{array}{l}
i_{m_{0}}^{0}-p-A(\uc)\\
i_{m_{1}}^{0}\\
\vdots\\
i_{m_{s}}^{0}
\end{array}
\right).
\end{equation}
Thus, $\Psi_{\um}(p;\uc)\times \{\uc\}\subset \Psi_{(\um,\un)}(p,q)$, and we can {\em slice} the set $\Psi_{(\um,\un)}(p,q)$ as
\begin{equation}\label{Sliced system solutions}
\Psi_{(\um,\un)}(p,q)=\bigcup_{\uc\in\Psi_{\un}(q)}\Psi_{\um}(p;\uc)\times \{\uc\}.
\end{equation}

We start studying system (\ref{System jotes}).
\begin{lemma}\label{Lemma:cota s}
If $q\geq j_{\l}^{0}+\dotsb+j_{\l}^{r}-1$ and $n_{0},\dotsc,n_{r}<\l$, then $\Psi_{\un}(q)=\emptyset$.
\end{lemma}
\begin{proof}
Let $\uc\in\ZZ^{r}$ be a solution of
\[
\left.
\begin{array}{r@{}l@{}}
j_{n_{1}+1}^{1}>c_{1}&\geq j_{n_{1}}^{1}\\
\vdots\\
j_{n_{r}+1}^{r}>c_{r}&\geq j_{n_{r}}^{r}
\end{array}
\right\}.
\]
In particular, $j_{n_{1}+1}^{1}+\dotsb+j_{n_{r}+1}^{r}>c_{1}+\dotsb+c_{r}\geq j_{n_{1}}^{1}+\dotsb+j_{n_{r}}^{r}$. On the other hand, the assumption on $q$ provides that 
$q-j_{\l}^{0}\geq j_{\l}^{1}+\dotsb+j_{\l}^{r}-1$. 
If $n_{0},\dotsc,n_{r}<\l$, then $j_{n_{0}+1}^{0},\dotsc,j_{n_{r}+1}^{r}<\infty$ and 
$q-j_{n_{0}+1}^{0}\geq q-j_{\l}^{0}\geq 
j_{\l}^{1}+\dotsb+j_{\l}^{r}-1\geq 
j_{n_{1}+1}^{1}+\dotsb+j_{n_{r}+1}^{r}-1$. Then, $q-j_{n_{0}+1}^{0}\geq j_{n_{1}+1}^{1}+\dotsb+j_{n_{r}+1}^{r}-1>c_{1}+\dotsb+c_{r}-1$. Hence, $q-j_{n_{0}+1}^{0}\geq c_{1}+\dotsb+c_{r}$ and it does not hold that $q-j_{n_{0}}^{0}\geq c_{1}+\dotsb+c_{r}>q-j_{n_{0}+1}^{0}$.
\end{proof}

\begin{lemma}\label{Lemma:cota t}
Let us fix $\un\in\{1,\dotsc,\l\}^{r}$ and $p,q\in\ZZ$ such that $p\geq i_{\l}^{0}+\dotsb+i_{\l}^{s}-a_{1}j_{1}^{1}-\dotsb-a_{r}j_{1}^{r}-1$. If $m_{0},\dotsc,m_{s}<\l$, then $\Psi_{\um}(p;\uc)=\emptyset$ for any solution $\uc\in\Psi_{\un}(q)$.
\end{lemma}
\begin{proof}
Let $(d_{1},\dotsc,d_{s})\in\ZZ^{s}$ be a solution of 
\[
\left.
\begin{array}{r@{}l@{}}
i_{m_{1}+1}^{1}>d_{1}&\geq i_{m_{1}}^{1}\\
\vdots\\
i_{m_{s}+1}^{s}>d_{s}&\geq i_{m_{s}}^{s}\\
\end{array}
\right\}.
\]
In particular, $i_{m_{1}+1}^{1}+\dotsb+i_{m_{s}+1}^{s}>d_{1}+\dotsb+d_{s}\geq i_{m_{1}}^{1}+\dotsb+i_{m_{s}}^{s}$. On the other hand, the assumption on $p$ implies that $p-i_{\l}^{0}\geq i_{\l}^{1}+\dotsb+i_{\l}^{s}-(a_{1}j_{1}^{1}+\dotsb+a_{r}j_{1}^{r})-1$. Let us fix any solution $(c_{1},\dotsc,c_{r})\in \Psi_{\un}(q)$.
If $m_{0},\dotsc,m_{s}<\l$, then $i_{m_{0}+1}^{0},\dotsc,i_{m_{s}+1}^{s}<\infty$ and 
$p-i_{m_{0}+1}^{0} \geq p - i_{\l} \geq i_{\l}^{1}+\dotsb+i_{\l}^{s}-(a_{1}j_{1}^{1}+\dotsb+a_{r}j_{1}^{r})-1\geq
i_{m_{1}+1}^{1}+\dotsb+i_{m_{s}+1}^{s}-(a_{1}j_{n_{1}}^{1}+\dotsb+a_{r}j_{n_{r}}^{r})-1\geq 
i_{m_{1}+1}^{1}+\dotsb+i_{m_{s}+1}^{s}-(a_{1}c_{1}+\dotsb+a_{r}c_{r})-1$. 
So,
$p-i_{m_{0}+1}^{0}\geq 
i_{m_{1}+1}^{1}+\dotsb+i_{m_{s}+1}^{s}-(a_{1}c_{1}+\dotsb+a_{r}c_{r})-1 > 
d_{1}+\dotsc+d_{s}-(a_{1}c_{1}+\dotsb+a_{r}c_{r})-1$.
Hence, $p-i_{m_{0}+1}^{0}\geq
d_{1}+\dotsc+d_{s}-(a_{1}c_{1}+\dotsb+a_{r}c_{r})$ and $(d_{1},\dotsc,d_{s})$ is not a solution of 
$p-i_{m_{0}}^{0}\geq 
d_{1}+\dotsb+d_{s} - (a_{1} c_{1}+\dotsb+a_{r}c_{r}) > 
p-i_{m_{0}+1}^{0}$.
\end{proof}

\begin{lemma}\label{Lemma:Polinomi i's}
Let us fix $\um\in\{1,\dotsc,\l\}^{s}$, $\un\in\{1,\dotsc,\l\}^{r}$ and $\uc\in\in\Psi_{\un}(q)$ a solution. If $p\geq i_{\l}^{0}+\dotsb+i_{\l}^{s}-a_{1}j_{1}^{1}-\dotsb-a_{r}j_{1}^{r}-1$, then $|\Psi_{\um}(p;\uc)|$ is a polynomial in $\CC[p,c_{1},\dotsc,c_{r}]$.
\end{lemma}
\begin{proof}
By Lemma \ref{Lemma:cota t}, we assume that $m_{k}=\l$ for at least one integer $0\leq k\leq s$.  Let us first assume that $m_{0}=m_{k_{1}}=\dotsc=m_{k_{t}}=\l$ for $0\leq t\leq s$. Up to permutation of indices we can assume that $k_{1}=1,\dotsc, k_{t}=t$. 
Then $|\Psi_{\um}(p;\uc)|$ is the number of integer solutions of the system (\ref{System i's}), which is equivalent to
\[
\left.
\begin{array}{r@{}l@{\hspace{-1mm}}}
p-i_{\l}^{0}\geq d_{1}+\dotsb+d_{s}- A(\uc)&> -\infty\\
\infty >d_{1}&\geq i_{\l}^{1}\\
\vdots\\
\infty > d_{t}&\geq i_{\l}^{t}\\
i_{m_{t+1}+1}^{t+1} > d_{t+1} &\geq i_{m_{t+1}}^{t+1}\\
\vdots\\
i_{m_{s}+1}^{s}>d_{s}&\geq i_{m_{s}}^{s}\\
\end{array}
\right\}.
\]
After the change of variables $e_{k}=d_{k}-i_{\l}^{k}$ for $1\leq k\leq t$, it is immediate to see that
\[
|\Psi_{\um}(p;\uc)|=
\hspace{-7mm}
\sum_{d_{t+1} =i_{m_{t+1}}^{t+1}}^{i_{m_{t+1}+1}^{t+1}-1}
\hspace{-4mm}
\dotsb
\hspace{-2mm}
\sum_{d_{s}=i_{m_{s}}^{s}}^{i_{m_{s}+1}^{s}-1}
\hspace{-3mm}
\binom{t+p-i_{\l}^{0}-i_{\l}^{1}-\dotsb-i_{\l}^{t}+\hspace{-1mm}A(\uc)-d_{t+1}-\dotsb-d_{s}}
{t}
\]
which is a polynomial. Indeed, fix any integer $(s-t)-$uple $(d_{t+1},\dotsc,d_{s})$ such that $i_{m_{t+1}+1}^{t+1} > d_{t+1} \geq i_{m_{t+1}}^{t+1},\dotsc,
i_{m_{s}+1}^{s}>d_{s}\geq i_{m_{s}}^{s}$. The bound on $p$ in the hypothesis, implies that $p-i_{\l}^{0}-i_{\l}^{1}-\dotsb-i_{\l}^{t}+A(\uc)-d_{t+1}-\dotsb-d_{s}\geq0$, which implies that
\[
\binom{t+p-i_{\l}^{0}-i_{\l}^{1}-\dotsb-i_{\l}^{t}+A(\uc)-d_{t+1}-\dotsb-d_{s}}
{t}
\]
is a polynomial in $\CC[p,c_{1},\dotsc,c_{r}]$, and the claim follows.

To finish the proof of the Lemma, we need to consider the case $1\leq m_{0}<\l$. Up to permutation of indices we assume that $m_{1}=\dotsc=m_{t}=\l$, but in this case $1\leq t \leq s$. After the change of variables $e_{1}:=d_{1}+\dotsb+d_{s}$, the system (\ref{System i's}) is equivalent to
\[
\left.
\begin{array}{@{}r@{}l@{}}
p-i_{m_{0}}^{0}\geq e_{1}- A(\uc)&>p-i_{m_{0}+1}^{0}\\
e_{1}-i_{\l}^{1}-d_{t+1}-\dotsb-d_{s} \geq d_{2}+\dotsb+d_{t}&>-\infty\\
\infty > d_{2}&\geq i_{\l}^{2}\\
\vdots\\
\infty > d_{t}&\geq i_{\l}^{t}\\
i_{m_{t+1}+1}^{t+1} > d_{t+1} &\geq i_{m_{t+1}}^{t+1}\\
\vdots\\
i_{m_{s}+1}^{s}>d_{s}&\geq i_{m_{s}}^{s}\\
\end{array}
\right\}.
\]
Thus, 
\[
\begin{array}{l>{\displaystyle}l}
|\Psi_{\um}(p;\uc)|\hspace{2mm}=&\hspace{-6mm}
\sum_{e_{1}=p-i_{m_{0}}^{0}+A(\uc)+1}^{p-i_{m_{0}^{0}+1}+A(\uc)}
\hspace{1mm}
\sum_{d_{t+1}=i_{m_{t+1}}^{t+1}}^{i_{m_{t+1}+1}^{t+1}-1}
\hspace{-4mm}
\dotsb\\
&\dotsb\sum_{d_{s}=i_{m_{s}}^{s}}^{i_{m_{s}+1}^{s}-1}
\binom{t-1+e_{1}-i_{\l}^{1}-\dotsb-i_{\l}^{t}-d_{t+1}-\dotsb-d_{s}}
{t-1}.
\end{array}
\]
Since $e_{1}-i_{\l}^{1}-\dotsb-i_{\l}^{t}-d_{t+1}-\dotsb-d_{s}>p-i_{m_{0}}^{0}-i_{\l}^{1}-\dotsb-i_{\l}^{t}-d_{t+1}-\dotsb-d_{s}+A(\uc)\geq p - i_{\l}^{0}-\dotsb-i_{\l}^{s}+a_{1}j_{1}^{1}+\dotsb+a_{r}j_{1}^{r}\geq0$, then
\[
P(e_{1})\hspace{2mm}=\hspace{-4mm}
\sum_{d_{t+1}= i_{m_{t+1}}^{t+1}}^{i_{m_{t+1}+1}^{t+1}-1}
\hspace{-4mm}
\dotsb
\sum_{d_{s}=i_{m_{s}}^{s}}^{i_{m_{s}+1}^{s}-1}
\binom{t-1+e_{1}-i_{\l}^{1}-\dotsb-i_{\l}^{t}-d_{t+1}-\dotsb-d_{s}}
{t-1}.
\]
is a polynomial in $\CC[e_{1}]$. Therefore, 
\[
|\Psi_{p}(\um,\uc)|\hspace{2mm}=\hspace{-4mm}
\sum_{e_{1}=p-i_{m_{0}}^{0}+A(\uc)+1}^{p-i_{m_{0}+1}^{0}+A(\uc)}\hspace{-8mm}P(e_{1})\hspace{1mm}=\hspace{-2mm}
\sum_{k=0}^{i_{m_{0}}^{0}-i_{m_{0}+1}^{0}}
\hspace{-4mm}P(k+p-i_{m_{0}}^{0}+A(\uc)),
\]
which is a finite sum of polynomials in $\CC[p,c_{1},\dotsc,c_{r}]$.
\end{proof}

\begin{definition}
The $n$th Bernoulli polynomial is defined recursively as
\[
B_n(x) = \sum_{k=0}^n \binom{n}{k} B_{n-k} x^k,
\]
where $B_{0}:=1$. In particular Bernoulli polynomials satisfy the well known Faulhaber identity:
\[
\sum_{k=0}^{q}k^{t}=\frac{B_{t+1}(q+1)-B_{t+1}(1)}{t+1}.
\]
\end{definition}

\begin{lemma}\label{Lemma:integracio sumatori}
If $P(q,e_{1},\dotsc,e_{k})\in\CC[q,e_{1},\dotsc,e_{k}]$ is a polynomial, then 
$\displaystyle{
\sum_{\substack{
e_{1}+\dotsc+e_{k}\leq q\\
e_{i}\geq0
}}
P(q,e_{1},\dotsc,e_{k})
}$ 
is a polynomial in $\CC[q]$.
\end{lemma}
\begin{proof}
We proceed by induction over $k$. If $k=1$, we want to see that the function $f(q):=\sum_{e_{1}=0}^{q}P(q,e_{1})$ is a polynomial. Notice that we can write $P(q,e_{1})=P_{0}(q)+P_{1}(q)e_{1}+\dotsb+P_{d}(q)e_{1}^{d}$ with $P_{i}\in\CC[q]$, where $d=\deg_{e_{1}}P$. Hence,
\[
f(q)=P_{0}(q)\sum_{e_{1}=0}^{q}1+P_{1}(q)\sum_{e_{1}=0}^{q}e_{1}+\dotsb+P_{d}(q)\sum_{e_{1}=0}^{q}e_{1}^{d},
\]
which is a $\CC[q]-$linear combination of the Bernoulli polynomials $B_{1}(q),\dotsc,$ $B_{d+1}(q)$.
Now we assume, by induction, that the result is true for some $k\geq 1$. We want to see that the function
\[
f(q):=\sum_{\substack{
e_{1}+\dotsc+e_{k+1}\leq q\\
e_{i}\geq0
}}
P(q,e_{1},\dotsc,e_{k+1})
\]
is a polynomial. We can write $f(q)$ as follows:
\[
\begin{array}{l>{\displaystyle}l}
f(q):=&
\hspace{-7mm}
\sum_{\substack{
e_{1}+\dotsc+e_{k+1}\leq q\\
e_{i}\geq0
}}
\hspace{-7mm}
P(q,e_{1},\dotsc,e_{k+1})=
\hspace{-6mm}
\sum_{\substack{
e_{1}+\dotsc+e_{k}\leq q\\
e_{i}\geq0
}}\hspace{-4mm}
\sum_{e_{k+1}=0}^{q-(e_{1}+\dotsb+e_{k})}
\hspace{-6mm}
P(q,e_{1},\dotsc,e_{k},e_{k+1})=\\[8mm]
&\hspace{-6mm}
\sum_{\substack{
e_{1}+\dotsc+e_{k}\leq q\\
e_{i}\geq0
}}\hspace{-4mm}
g(q,e_{1},\dotsc,e_{k}).
\end{array}
\]
By induction it is enough to see that $g(p,q,e_{1},\dotsc,e_{k})$ is a polynomial. We write $P(q,e_{1},\dotsc,e_{k},e_{k+1})=P_{0}(q,e_{1},\dotsc,e_{k})+P_{1}(q,e_{1},\dotsc,e_{k})e_{k+1}+\dotsb+
P_{d}(q,e_{1},\dotsc,e_{k})e_{k+1}^{d}$ with $P_{i}\in\CC[q,e_{1},\dotsc,e_{k}]$ and $d=\deg_{e_{k+1}}P$. Therefore,
\[
\begin{array}{l>{\displaystyle}l}
g(q,e_{1},\dotsc,e_{k})=&
P_{0}(q,e_{1},\dotsc,e_{k})\hspace{-7mm}
\sum_{e_{k+1}=0}^{q-(e_{1}+\dotsb+e_{k})}
\hspace{-7mm}
1+
P_{1}(q,e_{1},\dotsc,e_{k})\hspace{-7mm}
\sum_{e_{k+1}=0}^{q-(e_{1}+\dotsb+e_{k})}
\hspace{-7mm}e_{k+1}
+\dotsb\\
&\dotsb+
P_{d}(q,e_{1},\dotsc,e_{k})
\hspace{-7mm}
\sum_{e_{k+1}=0}^{q-(e_{1}+\dotsb+e_{k})}
\hspace{-7mm}e_{k+1}^{d}.
\end{array}
\]
Again, this is a $\CC[q,e_{1},\dotsc,e_{k}]-$linear combination of Bernoulli polynomials $B_{1}(q-(e_{1}+\dotsb+e_{k})+1),\dotsc,B_{d+1}(q-(e_{1}+\dotsb+e_{k})+1)$, and hence $g\in\CC[q,e_{1},\dotsc,e_{k}]$.
\end{proof}

We are now ready to prove the main result of the paper.

\begin{theorem}\label{Theorem:Bound hilbert polynomial}
Let $E$ be a multigraded reflexive module of rank $\l$ with associated filtrations as in Notation \ref{Notation:Reflexives on V(s,a1,...,ar)}. For $p\geq i_{\l}^{0}+\dotsb+i_{\l}^{s}-a_{1}j_{1}^{1}-\dotsb-a_{r}j_{1}^{r}-1$ and $q\geq j_{\l}^{0}+\dotsb+j_{\l}^{r}-1$, the Hilbert function $h_{\Gamma E}(p,q)$ of $\Gamma E$ is a polynomial.
\end{theorem}
\begin{proof}
For any integers $p,q$, the Hilbert function of $E$ at $(p,q)$ is $h_{E}(p,q)=\dim E_{(p,q)}=\sum_{m\in\ZZ^{s+r}}[\Gamma E_{(p,q)}]_{m}$. By Remark \ref{Remark:HilbertFunction} and (\ref{Sliced system solutions}),
\[
h_{\Gamma E}(p,q)=\sum_{
\substack{
	1\leq m_{0},\dotsc,m_{s}\leq\l\\
	1\leq n_{0},\dotsc,n_{r}\leq\l
}}|\Psi_{(\um,\un)}(p,q)|D(\um,\un)
\]
\[
=\sum_{1\leq n_{0},\dotsb,n_{r}\leq \l}\sum_{\bc\in\Psi_{\un}(q)}
\sum_{1\leq m_{0},\dotsc,m_{s}\leq \l}|\Psi_{\um}(p;\uc)|D(\um,\un)
\]
where $D(\um,\un):=\dim(F_{m_{0}}^{0}\cap \dotsb\cap F_{m_{s}}^{s}\cap G_{n_{0}}^{0}\cap \dotsb\cap G_{n_{r}}^{r})$. Rearranging the sum it yields:
\[
h_{\Gamma E}(p,q)=
\sum_{1\leq m_{0},\dotsc,m_{s}\leq \l}
\sum_{1\leq n_{0},\dotsc,n_{r}\leq \l}
\left[\sum_{\uc\in\Psi_{\un}(q)}
|\Psi_{\um}(p;\uc)|\right]
D(\um,\un).
\]
Hence, the proof reduces to see that if $p\geq i_{\l}^{0}+\dotsb+i_{\l}^{s}-a_{1}j_{1}^{1}-\dotsb-a_{r}j_{1}^{r}-1$ and $q\geq j_{\l}^{0}+\dotsb+j_{\l}^{r}-1$, the sum
\[
f(\um,\un,p,q):=\sum_{\uc\in\Psi_{\un}(q)}
|\Psi_{\um}(p;\uc)|
\]
is a polynomial in $\CC[p,q]$ for any $\um,\un$. Since $q\geq j_{\l}^{0}+\dotsb+j_{\l}^{r}-1$, Lemma \ref{Lemma:cota t} applies. In particular, we can assume that there is an integer $0\leq k\leq r$ such that $n_{k}=\l$, otherwise $f(\um,\un,p,q)=0$. Let us first suppose that $n_{0}=\l$ and reordering if necessary, $n_{1}=\dotsb=n_{t}=\l$ for $0\leq t\leq r$. Then the system (\ref{System jotes}) is equivalent to
\[
\left.
\begin{array}{rl@{}}
q-j_{\l}^{0}\geq c_{1}+\dotsb+c_{r}&> -\infty\\
\infty>c_{1}&\geq j_{\l}^{1}\\
\vdots\\
\infty>c_{t}&\geq j_{\l}^{t}\\
j_{n_{t+1}+1}^{t+1}>c_{t+1}&\geq j_{n_{t+1}}^{t+1}\\
\vdots\\
j_{n_{r}+1}^{r}>c_{r}&\geq j_{n_{r}}^{r}
\end{array}
\right\}.
\]
Applying the change of variables $e_{i}=c_{i}-j_{\l}^{i}$ for $1\leq i\leq t$ to this system we obtain
\[
f(\um,\un,p,q)=
\hspace{-3mm}
\sum_{c_{r}=j_{n_{r}}^{r}}^{j_{n_{r}+1}^{r}-1}
\hspace{-2mm}
\dotsb
\hspace{-2mm}
\sum_{c_{r}=j_{n_{t+1}}^{t+1}}^{j_{n_{t+1}+1}^{t+1}-1}
\hspace{-20mm}
\sum_{
\substack{
e_{1}+\dotsb+e_{t}=0\\
e_{i}\geq0
}
}^{\hspace{20mm}q - j_{\l}^{0}-\dotsb-j_{\l}^{t}-c_{t+1}-\dotsb-c_{r}}
\hspace{-24mm}
|\Psi_{p}(\um;e_{1}+j_{\l}^{1},\dotsc,e_{t}+j_{\l}^{t},c_{t+1},\dotsc,c_{r})|.
\]
By Lemma \ref{Lemma:Polinomi i's}, since $p\geq i_{\l}^{0}+\dotsb+i_{\l}^{s}-a_{1}j_{1}^{1}-\dotsb-a_{r}j_{1}^{r}-1$, then $|\Psi_{p}(\um;e_{1}+j_{\l}^{1},\dotsc,e_{t}+j_{\l}^{t},c_{t+1},\dotsc,c_{r})|$ is a polynomial in $\CC[p,e_{1},\dotsc,e_{t},$ $c_{t+1},\dotsc,c_{r}]$. Therefore, applying Lemma \ref{Lemma:integracio sumatori},
\[
P(p,q,c_{t+1},\dotsc,c_{r}):=
\hspace{-13mm}
\sum_{
\substack{
e_{1}+\dotsb+e_{t}\leq q - j_{\l}^{0}-\dotsb-j_{\l}^{t}-c_{t+1}-\dotsb-c_{r}\\
e_{i}\geq0
}
}
\hspace{-13mm}
|\Psi_{p}(\um;e_{1}+j_{\l}^{1},\dotsc,e_{t}+j_{\l}^{t},c_{t+1},\dotsc,c_{r})|
\]
is a polynomial in $\CC[p,q,c_{t+1},\dotsc,c_{r}]$, which implies that $f(\um,\un,p,q)\in\CC[p,q]$. To finish the proof we need to consider the case $1\leq n_{0}<\l$. Without loss of generality we can assume that $n_{1}=\dotsb=n_{t}=\l$ but now with $1\leq t\leq r$. Thus, after the change of variables $e_{1}=c_{1}+\dotsb+c_{r}$, the system (\ref{System jotes}) is equivalent to
\[
\left.
\begin{array}{r@{}l@{}}
q-j_{n_{0}}^{0}\geq e_{1} &> q-j_{n_{0}+1}^{0}\\
e_{1}-c_{t+1}-\dotsb-c_{r}-j_{\l}^{1}\geq c_{2}+\dotsb+c_{t}+&> -\infty \\
\infty > c_{2}&\geq j_{\l}^{2}\\
\vdots\\
\infty>c_{t}&\geq j_{\l}^{t}\\
j_{n_{t+1}+1}^{t+1}>c_{t+1}&\geq j_{n_{t+1}}^{t+1}\\
\vdots\\
j_{n_{r}+1}^{r}>c_{r}&\geq j_{n_{r}}^{r}
\end{array}
\right\}.
\]
Defining $e_{i}:=c_{i}-j_{\l}^{i}$ for $2\leq i\leq t$ we obtain $f(\um,\un,p,q)=$
\[
\begin{array}{@{}>{\displaystyle}l@{\hspace{-7mm}}>{\displaystyle}l@{}}
\hspace{-3mm}
\sum_{c_{r}=j_{n_{r}}^{r}}^{j_{n_{r}+1}^{r}-1}
\hspace{-2mm}
\dotsb
\hspace{-3mm}
\sum_{c_{r}=j_{n_{t+1}}^{t+1}}^{j_{n_{t+1}+1}^{t+1}-1}
\hspace{-4mm}
\sum_{\hspace{3mm}e_{1}=q-j_{n_{0}+1}^{0}+1}^{q-j_{n_{0}}^{0}}
\hspace{-7mm}&
\sum_{
\substack{
e_{2}+\dotsb+e_{t}=0\\
e_{i}\geq0
}
}^{ e_{1} - j_{\l}^{1}-\dotsb-j_{\l}^{t}-c_{t+1}-\dotsb-c_{r}}
\hspace{-14mm}
|\Psi_{p}(\um;e_{1}-e_{2}-j_{\l}^{2}-\dotsb\\
&\hspace{14mm}\dotsb-e_{t}-j_{\l}^{t},e_{2}+j_{\l}^{2},\dotsc,e_{t}+j_{\l}^{t},c_{t+1},\dotsc,c_{r})|.
\end{array}
\]
Analogously as in the previous case, this last equality proves that since $p\geq i_{\l}^{0}+\dotsb+i_{\l}^{s}-a_{1}j_{1}^{1}-\dotsb-a_{r}j_{1}^{r}-1$, then $f(\um,\un,p,q)\in\CC[p,q]$.
\end{proof}

The following corollary bounds the multigraded regularity index of a saturated multigraded module.

\begin{corollary}
Let $E$ be a multigraded module of rank $\l$ with associated filtrations as in Notation \ref{Notation:Reflexives on V(s,a1,...,ar)}. Then
\[
\omega_{\Gamma E}:=\left\{(p,q)\in\ZZ^{2}\mid 
\begin{array}{@{}l@{}l@{}}
p&\geq i_{\l}^{0}+\dotsb+i_{\l}^{s}-a_{1}j_{1}^{1}-\dotsb-a_{r}j_{1}^{r}-1\\
q&\geq j_{\l}^{0}+\dotsb+j_{\l}^{r}-1
\end{array}
\right\}
\]
is an upper bound of the multigraded regularity index $r.i.(\Gamma E)$ of $\Gamma E$.
\end{corollary}
\begin{proof}
It follows strightforward from Theorem \ref{Theorem:Bound hilbert polynomial} since if $(p,q)\in\omega_{\Gamma E}$, then $h_{\Gamma E}(p,q)=P_{E}(p,q)$.
\end{proof}

The following example shows the sharpness of this bound.
\begin{example}
Let $R=\CC[x_{0},x_{1},y_{0},y_{1}]$ be the Cox ring of the Hirzebruch surface $\cH_{3}$ as in Example \ref{Example:Hirzebruch}. Let us fix $\{e_{1},e_{2},e_{3}\}$ a basis of $\CC^{3}$ and we consider the following subspaces:
\[
\begin{array}{lll}
F^{0}_{1}=\langle 3e_{1}+3e_{2}+e_{3}\rangle &
F^{1}_{1}=\langle9e_{1}+4e_{2}+8e_{3}\rangle\\[1mm]
F^{0}_{2}=F^{0}_{1}+\langle4e_{1}+2e_{3}\rangle &
F^{1}_{2}=F^{1}_{1}+\langle2e_{1}+8e_{2}+8e_{2}\rangle\\[1mm]
G^{0}_{1}=\langle6e_{2}+3e_{3}\rangle&
G^{1}_{1}=\langle4e_{1}+4e_{3}\rangle\\[1mm]
G^{0}_{2}=G^{0}_{1}+\langle7e_{1}+e_{2}+3e_{3}\rangle&
G^{1}_{2}=G^{1}_{1}+\langle9e_{1}+8e_{2}\rangle.
\end{array}
\]
 Let $E$ be a (normalized) rank $3$ reflexive module with set of filtrations
\[
\begin{array}{ll}
E^{\rho_{0}}(-3,-1,0;F^{0}_{1},F^{0}_{2},\CC^{3})&E^{\rho_{1}}(-9,-3,0;F^{1}_{1},F^{1}_{2},\CC^{3})\\[1mm]
E^{\eta_{0}}(-4,-1,0;G^{0}_{1},G^{0}_{2},\CC^{3})&E^{\eta_{1}}(-2,-1,0;G^{1}_{1},G^{1}_{2},\CC^{3}).
\end{array}
\]
Using the package ToricVectorBundles of Macaulay2 \cite{M2}, we obtain the following values for $H^{1}(\cH_{3},\widetilde{E}(p,q))$ for $-4\leq q\leq 4$ and $2\leq q\leq 10$:
\[
\left(
\begin{array}{@{}ccccccccc@{}}
 3 & 2 & 1 & \tikzmark{left}{0} & 0 & 0 & 0 & 0 & 0 \\
 3 & 2 & 1 & 0 & 0 & 0 & 0 & 0 & 0 \\
 3 & 2 & 1 & 0 & 0 & 0 & 0 & 0 & 0 \\
 3 & 2 & 1 & 0 & 0 & 0 & 0 & 0 & 0 \\
 3 & 2 & 1 & 0 & 0 & 0 & 0 & 0 & 0 \\
 3 & 2 & 1 & 0 & 0 & 0 & 0 & 0 & \tikzmark{right}{0} \\
 11 & 10 & 9 & 8 & 8 & 8 & 8 & 8 & 8 \\
 24 & 24 & 24 & 24 & 24 & 24 & 24 & 24 & 24 \\
 31 & 33 & 35 & 37 & 39 & 41 & 43 & 45 & 47 \\
\end{array}
\right).
  \Highlight[first]
\]
The region highlighted in the figure corresponds to the bound in Theorem \ref{Theorem:Bound hilbert polynomial} for this case: $p\geq 5$ and $q\geq -1$.

Notice that the filtration corresponding to $E$ is {\em general}, that is a filtration such that any intersection of the form $\bigcap_{k=0}^{r}E^{k}_{n_{k}}$ has minimal dimension. Using Macaulay2, we have checked in many cases that the bound of Theorem \ref{Theorem:Bound hilbert polynomial} is sharp for modules with general filtrations.
\end{example}

\begin{corollary}
Let $E$ be a multigraded reflexive module of rank $\l$ and $\delta_{E}\in\ZZ^{2}$, such that $E(\delta_{E})$ is the normalized module as in Notation \ref{Notation:Reflexives on V(s,a1,...,ar)}. Then, for $p\geq -a_{1}(j_{1}^{1}-j_{\l}^{1})-\dotsb-a_{r}(j_{1}^{r}-j_{\l}^{r})-1$ and $q\geq -1$, the Hilbert function of $\Gamma E(\delta_{E})$, $h_{\Gamma E(\delta_{E})}$ coincides with the Hilbert polynomial $P_{E}$.
\end{corollary}
\begin{proof}
It is strightforward from Theorem \ref{Theorem:Bound hilbert polynomial} and the filtration of $E(\delta_{E})$ presented in Notation \ref{Notation:Reflexives on V(s,a1,...,ar)}.
\end{proof}

\begin{corollary}
Let $E$ be a multigraded reflexive module of rank $\l$ presented as a quotient
\[
\bigoplus_{i=1}^{b} R(\un_{i},\um_{i})\xrightarrow{\phi} \bigoplus_{j=1}^{c}R(\unu_j,\umu_j)\rightarrow \!E\rightarrow\! 0,\; 
\left\{
\begin{array}{@{\!}l@{}l}
(\un_{i},\um_{i})\!&=\!(n_{i}^{0}\!,\dotsc\!,n_{i}^{s},m_{i}^{0}\!,\dotsc\!,m_{i}^{r})\\
(\unu_{i},\umu_{i})\!&=\!(\nu_{i}^{0}\!,\dotsc\!,\nu_{i}^{s},\mu_{i}^{0}\!,\dotsc\!,\mu_{i}^{r}).
\end{array}
\right.
\]
If $p\geq -\min_{i}\{n_{i}^{0}\}-\dotsb-\min_{i}\{n_{i}^{s}\}+a_{1}\max_{j}\{\mu_{j}^{1}\}+\dotsb+a_{r}\max_{j}\{\mu_{j}^{r}\}-1$ and $q\geq -\min_{i}\{m_{i}^{0}\}-\dotsb-\min_{i}\{m_{i}^{r}\}-1$, then the Hilbert function of $E$ $h_{E}(p,q)$ is a polynomial.
\end{corollary}
\begin{proof}
We apply Theorem \ref{Theorem:Bound hilbert polynomial} with the bounds of Lemma \ref{Lemma:Bound reflexive filtration} for $i_{\l}^{k}$, $j_{\l}^{k}$ and $j_{1}^{k}$.
\end{proof}


\end{document}